\documentclass[12pt]{article}

\usepackage{graphicx, amssymb, latexsym, amsfonts, amsmath, lscape, amscd, amsthm, color, epsfig, mathrsfs, tikz, enumerate}
\usepackage{subcaption, hyperref}
\usepackage[normalem]{ulem} 

\newtheorem{theorem}{Theorem}[section]
\newtheorem{conjecture}[theorem]{Conjecture}
\newtheorem{corollary}[theorem]{Corollary}

\newtheorem{lemma}[theorem]{Lemma}
\newtheorem{proposition}[theorem]{Proposition}

\newtheorem{problem}[theorem]{Problem}
\newtheorem{definition}[theorem]{Definition}
\newtheorem{observation}[theorem]{Observation}

\begin{document}

\title{{\bf Density of $3$-critical signed graphs}}
\author{
{\sc Laurent Beaudou}$\,^{a}$, {\sc Penny Haxell}$\,^{b}$, {\sc Kathryn Nurse}$\,^{c}$  \\
{\sc Sagnik Sen}$\,^{d}$, and {\sc Zhouningxin Wang}$\,^{e}$ \\
\mbox{}\\
{\small $(a)$ Universit\'e Clermont Auvergne, France}\\
{\small $(b)$ University of Waterloo, Waterloo ON, Canada}\\
{\small $(c)$ Simon Fraser University, Burnaby, Canada}\\
{\small $(d)$ Indian Institute of Technology Dharwad, India}\\
{\small $(e)$ School of Mathematical Sciences and LPMC, Nankai University, Tianjin, China}\\
{\small Emails: lbeaudou@gmail.com; pehaxell@uwaterloo.ca; knurse@sfu.ca;}\\ 
{\small sen007isi@gmail.com; wangzhou@nankai.edu.cn}
}

\date{}

\maketitle

\begin{abstract}
We say that a signed graph is \emph{$k$-critical} if it is not $k$-colorable but every one of its proper subgraphs is $k$-colorable. Using the definition of colorability due to Naserasr, Wang, and Zhu~\cite{NWZ21} that extends the notion of circular colorability, we prove that every $3$-critical signed graph on $n$ vertices has at least $\frac{3n-1}{2}$ edges, and that this bound is asymptotically tight. It follows that every signed planar or projective-planar graph of girth at least $6$ is (circular) $3$-colorable, and for the projective-planar case, this girth condition is best possible. To prove our main result, we reformulate it in terms of the existence of a homomorphism to the signed graph $C_{3}^*$, which is the positive triangle augmented with a negative loop on each vertex.
\end{abstract}


\section{Introduction}
For a graph property $P$, we say that a graph $G$ is \emph{critical} for $P$ if every proper subgraph of $G$ satisfies $P$ but $G$ itself does not. Thus in particular, every graph that is not $k$-colorable contains a critical subgraph for $k$-colorability, and hence the study of critical graphs for coloring has been of key importance in the study of chromatic numbers.

In 2014, Kostochka and Yancey~\cite{KY14,KY14C} proved the following precise lower bound on the density of graphs that are critical for $3$-colorability. Note that the term \emph{$4$-critical} is used instead of ``critical for $3$-colorability" in~\cite{KY14,KY14C} but we avoid it here for consistency.

\begin{theorem}{\rm \cite{KY14,KY14C}}\label{thm:C3critical}
If a graph $G$ is critical for $3$-colorability, then $|E(G)|\geq \frac{5|V(G)|-2}{3}$.
\end{theorem}

Their short proof of this theorem in~\cite{KY14C}, coupled with a standard argument about the density of planar graphs, provided a new and elegant proof of the classical theorem of Gr\"otzsch that every triangle-free planar graph is $3$-colorable. In addition, Theorem~\ref{thm:C3critical} resolved the first open case of a well-known and decades-old conjecture of Ore~\cite{O67} on the density of critical graphs for $k$-colorability for every $k$. In~\cite{KY14}, Kostochka and Yancey proved a corresponding density bound for general $k$, thus solving Ore's Conjecture exactly or almost exactly in every case.

Our aim in this paper is to address the analogous density question in the setting of signed graphs.  A \emph{signed graph} $(G, \sigma)$ is a graph $G$ together with a signature $\sigma: E(G) \to \{+, -\}$. Thus a graph $G$ can be regarded as a signed graph $(G, +)$ with all edges being positive (i.e., assigned with $+$). One of the main notions distinguishing signed graphs from $2$-edge-colored graphs (whose edges are simply partitioned into two distinct types) is the operation of vertex switching. A \emph{switching} at a vertex $v$ of a signed graph corresponds to multiplying the signs of all (non-loop) edges incident to $v$ by $-$. Thus the sign of a loop is invariant under switching. Two signed graphs are said to be \emph{switching equivalent} if one can be obtained from the other by a sequence of vertex switchings. Accordingly, meaningful parameters of signed graphs should be invariant under vertex switching.

In the seminal paper~\cite{Z82c}, Zaslavsky introduced a natural definition of coloring of signed graphs with an even number of colors, half of which are positive, the other half negative. This notion has been extended to odd numbers by M\'a\v{c}ajov\'a, Raspaud, and Škoviera~\cite{MRS16} by introducing the color 0, which has a special status, and in a different and more symmetric way by Naserasr, Wang, and Zhu~\cite{NWZ21} who generalized the definition of circular coloring of graphs to the case of signed graphs. In this paper, we use the latter definition, which we now describe.

 In the graph setting, a
\emph{circular $\frac{p}{q}$-coloring} of a graph $G$ is a mapping $f: V(G)\to \{0,1, \ldots, p-1\}$ such that for each edge $uv$, $q\leq |f(u)-f(v)|\leq p-q$. We can think of this as assigning to each vertex a color chosen from a circular arrangement of $p$ colors, such that adjacent vertices receive colors that are at distance at least $q$ on the circle. This well-studied concept refines the usual definition of coloring, coinciding with the definition of $k$-coloring when $k=\frac{p}{q}$ is an integer. For signed graphs, given positive integers $p, q$ with $p\geq 2q$ and with $p$ even, a \emph{circular $\frac{p}{q}$-coloring} of a signed graph $(G,\sigma)$ is a mapping $\varphi:V(G)\to \{0, 1,
\ldots, p-1\}$ such that
\begin{itemize}
\item for each positive edge $uv$, $q\leq |\varphi(u)-\varphi(v)|\leq p-q$ and 
\item for each negative edge $uv$, either $|\varphi(u)-\varphi(v)|\leq \frac{p}{2}-q$ or $|\varphi(u)-\varphi(v)|\geq \frac{p}{2}+q$.
\end{itemize}
Intuitively, vertices adjacent via a positive edge should have colors
at distance at least $q$ on the $p$-cycle of colors as in the graph
case, while for vertices $u$ and $v$ adjacent via a negative edge, the color of $u$ should be at distance at least $q$ from the
\emph{antipodal color} $\frac{p}2+\varphi(v) \pmod p$ of $v$.
The \emph{circular chromatic number} of $(G,\sigma)$ is defined to be $$\chi_c(G, \sigma) = \min\Big\{\frac{p}{q}\mid (G, \sigma) \text{ admits a circular $\frac{p}{q}$-coloring}\Big\}.$$
It is not difficult to see that these definitions are invariant under vertex switching. The notion of criticality then extends in a natural way, and in particular, we say that a signed graph $(G,\sigma)$ is \emph{(circular) $3$-critical} if $\chi_c(G, \sigma)>3$ but
$\chi_c(H, \sigma)\leq3$ for every proper subgraph $H$ of $G$.

Our main result gives a signed graphs analogue of Theorem~\ref{thm:C3critical}.
\begin{theorem}\label{thm:maincolresult}
If $(G,\sigma)$ is a signed graph that is (circular) 3-critical, then $$|E(G)|\geq \frac{3|V(G)|-1}{2}.$$ 
\end{theorem}

Moreover, we show in Section~\ref{sec:Conclusion} that there is an infinite sequence of such signed graphs whose edge density is precisely
$\frac{3}{2}$. Hence our density bound is asymptotically tight.

One essentially immediate corollary of Theorem~\ref{thm:maincolresult} is the following result (see Subsection~\ref{subsec:context}), which is analogous to 
the simple derivation of Gr\"otzsch's theorem from Theorem~\ref{thm:C3critical}.

\begin{corollary}\label{cor:Planar+ProjectivePlanar}
Let $G$ be a planar or projective-planar graph of girth at least $6$. Then for every signature $\sigma$ on $G$, the signed graph $(G, \sigma)$ is (circular) $3$-colorable.
\end{corollary}

We show in Section~\ref{sec:Conclusion} that this girth bound is best possible for the class of signed projective-planar graphs. For the
planar case, it improves the previous best known bound of~$7$, proved
in~\cite{NSWX23}, but at present, we do not know whether the bound of~$6$ is tight. Constructions given in~\cite{NWZ21,KNNW23} show that the correct bound cannot be smaller than~$5$.

\paragraph{Homomorphisms.}
In fact, it will be more natural for us to formulate and prove
Theorem~\ref{thm:maincolresult} in terms of homomorphisms. Recall that
in the graph setting, a \emph{homomorphism} of a graph $G$ to a graph $H$ is a vertex mapping $\varphi: V(G)\to V(H)$ such that adjacency is preserved. This is a generalization of the definition of coloring, for example, any proper vertex $k$-coloring of $G$ can be viewed as a homomorphism of $G$ to the complete graph $K_{k}$, and it is well known that a graph admits a circular $\frac{2k+1}{k}$-coloring if and only if it admits a homomorphism to the odd cycle $C_{2k+1}$. For a given graph $H$, a graph $G$ is called \emph{$H$-critical} if $G$ does not admit a homomorphism to $H$, but every proper subgraph of $G$ does.

The definition of homomorphism extends to signed graphs $(G,\sigma)$ as follows. For a closed walk $W$ in $G$, the \emph{sign} of $W$ is the product of the signs of all the edges in $W$ (allowing repetition). A \emph{homomorphism} of $(G, \sigma)$ to another signed graph $(H, \pi)$ is a mapping of $V(G)$ to $V(H)$ such that both the adjacency and the signs of all closed walks are preserved. If there exists a homomorphism of $(G, \sigma)$ to $(H, \pi)$, then we write $(G, \sigma)\to (H, \pi)$. Again it is easy to see that the existence of a homomorphism is invariant under switching. The definition of criticality also extends in the natural way: given a signed graph $(H, \pi)$, a signed graph $(G, \sigma)$ is said to be \emph{$(H, \pi)$-critical} if $(G, \sigma)$ does not admit
a homomorphism to $(H, \pi)$, but every proper subgraph of $(G, \sigma)$ does. (More accurately, with certain girth conditions, see Definition~\ref{def:H-critical}.)

Our main interest in this paper is in (circular) $3$-coloring of signed graphs, which by definition is the case $\ell=3$ of circular $\frac{2\ell}{\ell-1}$-coloring. This sequence of rationals turns out to be of special interest and importance, in that (analogously to the graph case) it is closely related to the existence of homomorphisms of signed graphs to signed cycles. We write $C^*_{\ell}$ for a signed cycle of length $\ell$ with an odd number of positive edges, together with negative loops at each vertex, see Figure~\ref{fig:C_-k}. (Note that for fixed $\ell$, all such cycles are switching-equivalent.)

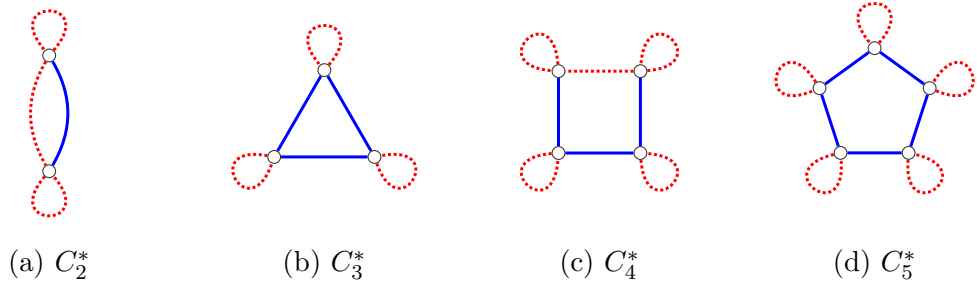
\begin{figure}[htbp]
	\centering
	\begin{subfigure}[t]{.22\textwidth}
		\centering
		\begin{tikzpicture}
			[scale=.22]
			\foreach \i in {1,2}
			{
				\draw[rotate=180*(\i-1)] (0, 3.5) node[circle, draw=black!80, inner sep=0mm, minimum size=1.7mm] (x_\i){};
			}
			
			\foreach \i/\j in {1/2}		
			{
				\draw [densely dotted, line width=0.4mm, red] (x_\i) edge[bend right] (x_\j);
			}
			
			\foreach \i/\j in {2/1}		
			{
				\draw [line width=0.4mm, blue] (x_\i) edge[bend right] (x_\j);
			}
			
			\foreach \i in {1,2}
			{
				\draw[rotate=180*(\i-1)] [densely dotted, line width=0.4mm, red] (x_\i) .. controls (3,7) and (-3,7) .. (x_\i);
			}	
		\end{tikzpicture}
		\caption{$C^*_{2}$}
	\end{subfigure}
	\begin{subfigure}[t]{.22\textwidth}
		\centering
		\begin{tikzpicture}
			[scale=.22]
			\foreach \i in {1,2,3}
			{
				\draw[rotate=120*(\i-1)] (0, 3.5) node[circle, draw=black!80, inner sep=0mm, minimum size=1.7mm] (x_\i){};
			}
			
			\foreach \i/\j in {1/2, 2/3, 3/1}
			{
				\draw  [line width=0.4mm, blue] (x_\i) -- (x_\j);
			}
			
			\foreach \i in {1,2,3}
			{
				\draw[rotate=120*(\i-1)] [densely dotted, line width=0.4mm, red] (x_\i) .. controls (3,7) and (-3,7) .. (x_\i);
			}	
		\end{tikzpicture}
		\caption{$C^*_{3}$}
	\end{subfigure}
	\begin{subfigure}[t]{.22\textwidth}
		\centering
		\begin{tikzpicture}
			[scale=.22]
			\foreach \i in {1,2,3,4}
			{
				\draw[rotate=45*(2*\i-1)] (0, 3.5) node[circle, draw=black!80, inner sep=0mm, minimum size=1.7mm] (x_\i){};
			}
			
			\foreach \i/\j in {1/2, 2/3, 3/4}
			{
				\draw  [line width=0.4mm, blue] (x_\i) -- (x_\j);
			}
			\foreach \i/\j in {4/1}
			{
				\draw  [densely dotted, line width=0.4mm, red] (x_\i) -- (x_\j);
			}
			
			\foreach \i in {1,2,3,4}
			{
				\draw[rotate=45*(2*\i-1)] [densely dotted, line width=0.4mm, red] (x_\i) .. controls (3,7) and (-3,7) .. (x_\i);
			}	
		\end{tikzpicture}
		\caption{$C^*_{4}$}
	\end{subfigure}
	\begin{subfigure}[t]{.22\textwidth}
		\centering
		\begin{tikzpicture}
			[scale=.22]
			\foreach \i in {1,2,3,4,5}
			{
				\draw[rotate=72*(\i-1)] (0, 3.5) node[circle, draw=black!80, inner sep=0mm, minimum size=1.7mm] (x_\i){};
			}
			
			\foreach \i/\j in {1/2, 2/3, 3/4, 4/5, 5/1}
			{
				\draw  [line width=0.4mm, blue] (x_\i) -- (x_\j);
			}
			
			\foreach \i in {1,2,3,4,5}
			{
				\draw[rotate=72*(\i-1)] [densely dotted, line width=0.4mm, red] (x_\i) .. controls (3,7) and (-3,7) .. (x_\i);
			}	
		\end{tikzpicture}
		\caption{$C^*_{5}$}
	\end{subfigure}
	\caption{Signed graphs $C^*_{\ell}$. Solid blue edges are positive, dashed red edges are negative.}
	\label{fig:C_-k}
\end{figure}

The following fact from~\cite{NW22+} gives a characterization of circular $\frac{2\ell}{\ell-1}$-colorable signed graphs in terms of 
 homomorphisms.

\begin{proposition}{\rm \cite{NW22+}}\label{Prop:HomoToCycleAndCircularColoring}
A signed graph admits a circular $\frac{2\ell}{\ell-1}$-coloring if and only if it admits a homomorphism to $C^*_{\ell}$.
\end{proposition}

Proposition~\ref{Prop:HomoToCycleAndCircularColoring} implies the reformulation of Theorem~\ref{thm:maincolresult} that will be our main focus from now on.

\begin{theorem}\label{thm:main result}
Every $C^*_3$-critical signed graph $(G,\sigma)$ satisfies $|E(G)|\geq \frac{3|V(G)|-1}{2}$.
\end{theorem}

The rest of the paper is organized as follows. In the next subsection, we give the (simple) proof of Corollary~\ref{cor:Planar+ProjectivePlanar}, and also outline how our main results relate to other previous work on graphs and signed graphs. While this material is not essential to the understanding of this paper, it provides further motivation and places our results in a broader context. In Section~\ref{sect:prelim}, we give preliminary background on signed graphs and homomorphisms of signed graphs. In particular, in Section~\ref{sec:PropertiesOfCriticalGraphs}, we provide some basic properties of $C^*_3$-critical signed graphs. In Section~\ref{sec:MainProof}, we use the potential method employed in~\cite{KY14}, adapted to our setting, to find more forbidden configurations in the minimum counterexample of our main theorem (Theorem~\ref{thm:main result}) and use the discharging technique to complete the proof. In Section~\ref{sec:Conclusion}, we show that the edge density bound of Theorem~\ref{thm:main result} is asymptotically tight, and the girth bound of Corollary~\ref{cor:Planar+ProjectivePlanar} for the class of signed projective-planar graphs is tight. We also pose some open questions there.

\subsection{Further Context}\label{subsec:context}
As previously noted, results such as Theorem~\ref{thm:C3critical} and Theorem~\ref{thm:main result} have direct implications for colorings (or more generally homomorphisms) of graphs whose densities are bounded above, for example, graphs embedded on surfaces or graphs with large girth. We see this explicitly in the following proof of Corollary~\ref{cor:Planar+ProjectivePlanar}. 

\smallskip
\noindent
{\em Proof of Corollary~\ref*{cor:Planar+ProjectivePlanar}.} 
Let ${G= (V, E)}$ be a planar or projective-planar graph of girth
at least $6$, and $\sigma$ be a signature on $G$. If $(G, \sigma)$ is not $3$-colorable, then by Proposition~\ref{Prop:HomoToCycleAndCircularColoring} we may assume without loss of generality that it is $C_{3}^*$-critical. Consider a plane or
projective-plane embedding of $G$, and let its set of faces be denoted by $F$. Euler's formula states that $|V|-|E|+|F| = 2-g$ where $g$ is the genus of the surface in which the graph is embedded ($0$ for the plane
and $1$ for the projective plane). The girth condition applied to the embedding gives that $|E| \ge 3|F|$. Hence we obtain that $|E| \leq \frac{3|V| - 3(2-g)}{2}$, which contradicts Theorem~\ref{thm:main result}. \hfill $\Box$

\smallskip
In classical graph theory, one major motivation for proving lower bounds on the density of critical graphs was Jaeger's famous Circular Flow Conjecture~\cite{J84,J88}, which was recently disproved for $k\geq 3$ by Han, Li, Wu, and Zhang~\cite{HLWZ18}. However, its planar restriction remains open and can be stated as follows.

\begin{conjecture}{\rm \cite{J84}}\label{conj:Jaeger-restriction}
For any integer $k\geq 1$, every planar graph of girth at least $4k$ admits a homomorphism to $C_{2k+1}$.
\end{conjecture}

For general $k$, the best result is due to Lov\'asz, Thomassen, Wu, and Zhang~\cite{LTWZ13}, that the girth condition $6k$ is sufficient. For small values of $k$, tighter results are known. The case $k=1$ is simply Gr\"{o}tzsch's theorem~\cite{G58}; for $k=2$, it has been verified by Dvo\v{r}\'ak and Postle~\cite{DP17} for the girth
condition $10$; for $k=3$, the best-known girth bound of $16$ has very recently been achieved by Postle and Smith-Roberge~\cite{PS22}. The same results for $k=2,3$ were independently obtained by Cranston and Li~\cite{CL20} using the notion of flows. 
The results of~\cite{DP17} and~\cite{PS22} are each proved by establishing lower bounds on the density of $C_{2k+1}$-critical graphs, as follows.

\begin{theorem}{\rm \cite{DP17}}
Every $C_{5}$-critical graph $G$ except $C_{3}$ satisfies $|E(G)|\geq \frac{5|V(G)|-2}{4}$.
\end{theorem}

\begin{theorem}{\rm \cite{PS22}}
Every $C_{7}$-critical graph $G$ except $C_{3}$ and $C_{5}$ satisfies $|E(G)|\geq \frac{17|V(G)|-2}{15}$.
\end{theorem}

The general problem of finding the best possible lower bound on the edge density of $C_{2k+1}$-critical graphs has been studied extensively in the literature, and we refer to the two papers above and the references therein.

In the setting of signed graphs, the \emph{girth} of a signed graph $(G,\sigma)$ is defined as the length of a shortest cycle in $G$, and its \emph{negative-girth} as the length of its shortest negative cycle. Parallel to the graph case, the following natural questions have been addressed in the literature.

\begin{enumerate}
    \item What is the edge density of $C_{\ell}^*$-critical signed graphs?
    \item {What is the smallest integer $f(\ell)$ such that every signed planar graph of girth at least $f(\ell)$ admits a homomorphism to $C_{\ell}^*$?}
\end{enumerate}

Naserasr, Wang, and Zhu~\cite{NWZ21} have proved that every signed planar graph of girth at least $4$ admits a
homomorphism to $C^*_{2}$, and the girth bound is best possible due to a result of Kardo\v{s} and Narboni~\cite{KN21}. Moreover, by Proposition~\ref{Prop:HomoToCycleAndCircularColoring}, the circular chromatic number bound $4$ of such signed graphs is asymptotically tight, as there is a sequence of signed bipartite planar simple graphs whose circular chromatic number is approaching $4$~\cite{KNNW23}.
Regarding negative cycles as homomorphism targets, the signed cycle of length $\ell$ with an odd number of negative edges, written $C_{-\ell}$, has also been studied. For $\ell=4$, when restricted to bipartite graphs, the following results have been established. 
\begin{theorem}{\rm \cite{NPW22}}
\begin{itemize}
    \item Every $C_{-4}$-critical signed graph $(G,\sigma)$ except one signed graph on $7$ vertices and with $9$ edges satisfies that $|E(G)|\geq \frac{4|V(G)|}{3}$. 
    \item Every signed bipartite planar graph of negative-girth at least $8$ admits a homomorphism to $C_{-4}$. Moreover, the negative-girth bound is the best possible.
\end{itemize}
\end{theorem}
Thus our Theorem~\ref{thm:main result} and Corollary~\ref{cor:Planar+ProjectivePlanar} further contribute to this line of investigation.
In particular,
Corollary~\ref{cor:Planar+ProjectivePlanar} can be viewed as addressing the most basic case of a signed graph analogue of Conjecture~\ref{conj:Jaeger-restriction}.

\section{Preliminaries}\label{sect:prelim} 
In this paper, all graphs are finite and may have multiple edges or loops. If the signature of a signed graph $(G, \sigma)$ is understood from the context, or its particular knowledge is irrelevant, we use the simplified notation $\hat{G}$ to denote it. We denote the \emph{underlying graph} of a signed graph $\hat{G}=(G, \sigma)$ by $G$. For the figures, we use a blue solid line to represent a positive edge, a red dashed line to represent a negative edge, and a gray line to represent an unsigned edge. A \emph{digon} is two parallel edges with different signs. We say $(H, \pi)$ is a \emph{subgraph} of $(G, \sigma)$ if $H$ is a subgraph of $G$ and $\pi=\sigma|_{H}$.

We use $v(G)$ to denote the number of vertices of $G$ and $e(G)$ to denote the number of edges of $G$. We say a vertex is a \emph{distance-two neighbor} of another vertex if they are connected by a path of length $2$ whose internal vertex is of degree $2$.

A \emph{$k$-vertex} is a vertex having degree $k$ and a \emph{$k^+$-vertex} is a vertex of degree at least $k$. A \emph{$k_{\geq \ell}$-vertex} (or, \emph{$k_{\leq \ell}$-vertex}) is a $k$-vertex with at least (respectively, at most) $\ell$ neighbors of degree $2$ and a \emph{$k_{\ell}$-vertex} is a $k$-vertex with exactly $\ell$ neighbors of degree $2$. Other standard notions follow \cite{W96}.

\subsection*{Homomorphisms of signed graphs} 
Switching a subset $S$ of vertices of $(G,\sigma)$ amounts to toggling the sign of all the edges of the edge-cut $[S, V(G) \setminus S]$. Two signed graphs $(G,\sigma)$ and $(G, \sigma')$, or alternatively, the two signatures $\sigma$ and $\sigma'$ on $G$, are switching equivalent if we can obtain $(G, \sigma')$ from $(G, \sigma)$ by switching at an edge-cut. Note that switching at an edge-cut does not change the signs of any closed walk (or cycle). One of the earliest results \cite{Z82} proved in the theory of signed graphs characterizes equivalent signed graphs using the sign of their cycles (or closed walks).

\begin{lemma}\label{lem:SwitchingEquivalent}{\rm \cite{Z82}}
Two signed graphs $(G,\sigma)$ and $(G, \sigma')$ are switching equivalent if and only if each cycle has the same sign in both signed graphs.  
\end{lemma}

Recall that a homomorphism of $(G,\sigma)$ to $(H, \pi)$ is a mapping $f: V(G) \to V(H)$ such that the adjacency and the signs of closed walks are preserved. A homomorphism of $(G, \sigma)$ to $(H, \pi)$ is said to be \emph{edge-sign preserving} if, furthermore, it preserves the signs of edges.

\begin{proposition}\label{prop:Homo}{\rm \cite{NRS15}}
A signed graph $(G, \sigma)$ admits a homomorphism to $(H, \pi)$ if and only if there exists a switching-equivalent signature $\sigma'$ such that $(G, \sigma')$ admits an edge-sign preserving homomorphism to $(H, \pi)$.
\end{proposition}

We have noted before that based on the sign of the cycles and the parity of their lengths, there are four types of closed walks in signed graphs: positive odd closed walk (type $01$), negative odd closed walk (type $11$), positive even closed walk (type $00$) and negative even closed walk (type $10$). We denote by $g_{_{ij}}(G, \sigma)$ for $ij\in \mathbb{Z}^2_2$ the length of a shortest closed walk of type $ij$ in a signed graph $(G, \sigma)$. The next lemma provides a necessary condition for a signed graph to admit a homomorphism to another. 

\begin{lemma}{\rm \cite{NSZ21}}
If $(G, \sigma)\to (H, \pi)$, then $g_{_{ij}}(G, \sigma)\geq g_{_{ij}}(H, \pi)$ for $ij\in \mathbb{Z}^2_2$.
\end{lemma}

It is easy to observe that if a signed graph $(G, \sigma)$ is $(H, \pi)$-critical and there exists $ij\in \mathbb{Z}^2_2$ such that $g_{_{ij}}(G, \sigma)\leq g_{_{ij}}(H, \pi)$, then $(G, \sigma)$ is just a signed cycle of type $ij$. To eliminate the trivial case, we use the notion of $(H, \pi)$-critical signed graph defined as follows:

\begin{definition}\label{def:H-critical}{\rm \cite{NPW22}}
A signed graph $\hat{G}$ is \emph{$\hat{H}$-critical} if for $ij\in \mathbb{Z}^2_2$, $g_{_{ij}}(\hat{G})\geq g_{_{ij}}(\hat{H})$, $\hat{G}$ admits no homomorphism to $\hat{H}$ but any proper subgraph of $\hat{G}$ does.
\end{definition}

In particular, this means any $C^*_3$-critical signed graph has no digon and no positive loop.

\subsection*{Circular colorings of signed graphs} 

The notion of the circular coloring of signed graphs is a refinement of both the notions of $0$-free colorings of signed graphs and circular colorings of simple graphs. Now we are going to show how homomorphism captures the notion of circular $\frac{p}{q}$-coloring of signed graphs for any rational number $\frac{p}{q}$. To do so, we need a special family of signed graphs. 

\begin{definition}{\rm \cite{NWZ21}}
Given two positive integers $p$ and $q$ with $p$ being even, the \emph{circular $\frac{p}{q}$-clique}, denoted $K^s_{p;q}$, is a signed graph having the set of vertices $\{0, 1, \cdots, p-1\}$, and edges and signature as follows: (1) $ij$ is a positive edge if $q \leq |i-j| \leq p-q$, (2) $ij$ is a negative edge if either $|i-j| \leq \frac{p}{2}-q$ or $|i-j| \geq \frac{p}{2}+q$. 
\end{definition}

The signed graph $K^s_{p;q}$ contains a negative loop at each vertex, and, moreover, it contains a digon if and only if $\frac{p}{q}\geq 4$. 

Note that in $K^s_{p;q}$, each vertex $i$ and its antipodal vertex $\bar{i}=i+\frac{p}{2}$ (taken modulo $p$) has exactly an opposite neighborhood, that is to say, a vertex is adjacent to $i$ by a positive edge while it is adjacent to $\bar{i}$ by a negative edge. We may switch at $\{\frac{p}{2},\ldots,p-1\}$
and identify each of them with their antipodes, and the resulting signed graph is denoted by $\hat{K}^s_{p;q}$. Such $\hat{K}^s_{p;q}$ has exactly $\frac{p}{2}$ vertices. 

Given any positive rational number $\frac{p}{q}$, the circular $\frac{p}{q}$-clique $K^s_{p;q}$ and its switching core $\hat{K}^s_{p;q}$ are put into our context in the following proposition. 

\begin{proposition}{\rm \cite{NWZ21}}
Given a signed graph $(G, \sigma)$, the following statements are equivalent:
\begin{itemize}
    \item $(G, \sigma)$ admits a circular $\frac{p}{q}$-coloring;
    \item $(G, \sigma)$ admits an edge-sign preserving homomorphism to $K^s_{p;q}$;
    \item $(G, \sigma)$ admits a homomorphism to $\hat{K}^s_{p;q}$.
\end{itemize}
\end{proposition}

We note that $\hat{K}^s_{2\ell;\ell-1}$ is switching isomorphic to $C^*_\ell$.

\subsection{Properties of $C_3^*$-critical signed graphs}\label{sec:PropertiesOfCriticalGraphs}

We say a triangle is a graph that is a cycle $C_3$ of length three. We denote by $C_3^*$ the signed graph in Figure~\ref{fig:C3 star} which is a positive triangle with a negative loop at each vertex. 

Recall that a signed graph $\hat{G}$ is \emph{$C_3^*$-critical} if the following three conditions are satisfied: there is no digon or positive loop in $\hat{G}$, $\hat{G} \not\rightarrow C_3^*$, and $\hat{G}' \rightarrow C_3^*$ for any proper subgraph $\hat{G}'\subsetneq \hat{G}$. First we give an example, depicted in Figure~\ref{fig:W}, which is $C_3^*$-critical and  satisfies the condition $e(G)=\frac{3v(G)-1}{2}$.

\begin{figure}[ht]
    \centering 
\begin{minipage}[t]{.4\textwidth}
    \centering
    \begin{tikzpicture}[scale=.28]
   \foreach \i/\j in {1,2,3}
  {
    \draw[rotate=120*(\i)] (0,4) node[circle, draw=black!80, inner sep=0mm, minimum size=3.3mm] (x_{\i}){\footnotesize $x_{\scriptsize \i}$};
  }
	
\foreach \i/\j in {1/2,2/3,3/1}
  {
    \draw[line width=0.4mm, blue] (x_{\i}) -- (x_{\j});
  }

  \foreach \i/\j in {1,2,3}
  {
    \draw[rotate=120*(\i)] [densely dotted, line width=0.4mm, red] (x_{\i}) .. controls (3,7) and (-3,7) .. (x_{\i});
  }	
  \end{tikzpicture}
        \caption{$C^*_3\equiv \hat{K}^{s}_{6;2}$}
            \label{fig:C3 star}
   \end{minipage}
\begin{minipage}[t]{.45\textwidth}
\centering
       \begin{tikzpicture}[scale=.28]
       \centering
	\draw (0,0) node[circle, draw=black!80, inner sep=0mm, minimum size=3.3mm] (x_{0}){\footnotesize $v_{\scriptsize  0}$};
	\draw (0,-2.5) node[circle, draw=black!80, inner sep=0mm, minimum size=3.3mm] (x_{4}){\footnotesize $v_{\scriptsize 4}$};
\foreach \i/\j in {1,2,3}
  {
    \draw[rotate=120*(\i)] (0,5) node[circle, draw=black!80, inner sep=0mm, minimum size=3.3mm] (x_{\i}){\footnotesize $v_{\scriptsize \i}$};
  }
	
\foreach \i/\j in {2/3,3/1}
  {
    \draw[line width=0.4mm, blue] (x_{\i}) -- (x_{\j});
  }
 \foreach \i in {1,2,3}
  {
    \draw[line width=0.4mm, blue] (x_{\i}) -- (x_{0});
  }
  \draw[densely dotted, line width=0.4mm, red] (x_{1}) -- (x_{4});
  \draw[line width=0.4mm, blue] (x_{2}) -- (x_{4});
  \draw[rotate=120*(2)] [densely dotted, line width=0.4mm, white] (x_{2}) .. controls (3,7) and (-3,7) .. (x_{2});
  \end{tikzpicture}
        \caption{$C^*_3$-critical signed graph $\hat{W}$}
            \label{fig:W}
  \end{minipage}
\end{figure}

\begin{lemma}\label{lem:hatW}
The signed graph $\hat{W}$ is $C_3^*$-critical. 
\end{lemma}

\begin{proof}
Suppose for contradiction that $\hat{W}\to C_3^*$. By Proposition~\ref{prop:Homo}, there is a switching-equivalent signature $\sigma'$ of $\hat{W}$ such that $(W, \sigma')$ admits an edge-sign preserving homomorphism to $C^*_3$. We first observe that under $\sigma'$ any negative $4$-cycle contains only one negative edge and any positive cycle contains no negative edges. Subject to these two conditions, $\sigma'$ is unique as drawn in Figure~\ref{fig:W}. Hence, by examining the subgraph $\hat{W}- v_4$, in any edge-sign preserving homomorphism $\varphi$ of $(W, \sigma')$ to $C^*_3$, $\varphi(v_1)=\varphi(v_2)$. But then together with $\varphi(v_4)$, it would form a digon, which does not exist in $C^*_3$, a contradiction. Therefore, $\hat{W}\not\to C_3^*$. Finally, it is easy to see that any proper subgraph of $\hat{W}$ admits a homomorphism to $C^*_3$. Thus $\hat{W}$ is $C^*_3$-critical.
\end{proof}

In the arguments that follow we will employ a general technique to ``color'' a signed graph $\hat{G}$ by extending a ``pre-coloring'' of its subgraph $\hat{H}$. First, we assume that there exists an edge-sign preserving homomorphism of $\hat{H}$ to $C_3^*$ under the signature of $\hat{G}$. Once we fix the homomorphism of $\hat{H}$ to $C_3^*$, we never again switch at the vertices of $\hat{H}$. To extend this homomorphism, we may switch at the vertices in $V(G) \setminus V(H)$. In this setting, it makes sense to speak of the sign of a path if both ends of a path are fixed in $\hat{H}$. The \emph{sign} of a path is then the product of the signs of all of its edges. Motivated by this, in the sequel, we use figures with round or square vertices to denote properties of the coloring and structure: we use a round vertex to denote a vertex that is not pre-colored, at which we allow switching, and whose degree is shown in the figure; We use a square vertex to denote a vertex which is pre-colored, at which we do not allow switching, which may have neighbors not drawn in the figure, and, moreover, which may not be distinct from other square vertices.

\begin{observation}\label{obs:PathExtension}
Let $\hat{P}$ be a signed path with the endpoints $x$ and $y$, which contains at most one negative edge. Let $S_x, S_y\subseteq V(C^*_3)$. Let 
$\varphi: \{x, y\} \to V(C^*_3)$ be such that $\varphi(x)\in S_x$ and $\varphi(y)\in S_y$. The mapping $\varphi$ can be extended to an edge-sign preserving homomorphism of $\hat{P}$ to $C^*_3$ unless one of the following conditions is satisfied: 
\begin{enumerate}
\item $\hat{P}$ is either a positive edge or a negative path of length $2$, $S_x=S_y$ and $|S_x|=1$; 
\item $\hat{P}$ is a negative edge and $S_x\cap S_y=\emptyset$.
\end{enumerate}
\end{observation}

To justify this observation, note that when the above conditions are not satisfied $\hat{P}$ has at least two positive edges, which affords enough flexibility in the mapping.

A \emph{theta graph} is a simple graph that is the union of three internally disjoint paths that have the same two end vertices.

\begin{lemma}\label{lem:theta}
Every signed theta graph $\hat{\Theta}$ admits a homomorphism to $C^*_3$ and is therefore not $C^*_3$-critical.
\end{lemma}

\begin{proof}
Among the three paths of a signed theta graph $\hat{\Theta}$, two of them, say $\hat{P}_1,\hat{P}_2$ with $v(\hat{P}_1) \geq v(\hat{P}_2)$, have the same parity of the number of positive edges. For each negative edge $e = uv$ of those paths, identify $u$ and $v$. Now it is easy to see that $\hat{P}_1 \rightarrow \hat{P}_2$. Therefore, $\hat{\Theta} \to C_3^*$ if and only if $\hat{\Theta} - E(\hat{P}_1) \to C_3^*$. But $\hat{\Theta}-E(\hat{P}_1)$ is a signed cycle (which is not a digon), and therefore admits a homomorphism to $C^*_3$.
\end{proof}

In fact, Lemma~\ref{lem:theta} holds more generally: no signed theta graph $\hat{\Theta}$ is $C^*_\ell$-critical, and every signed theta graph which does not violate the girth conditions admits a homomorphism to $C^*_\ell$.

\begin{lemma}\label{lem: forbidden edge}
No $C_3^*$-critical signed graph contains an edge of the following type: loop-edge, parallel-edge, or cut-edge.  
\end{lemma}
\begin{proof}
Let $\hat{G}$ be a $C_3^*$-critical signed graph, and let $e \in E(\hat{G})$. By definition of criticality, $\hat{G} \not\to C_3^*$ but $\hat{G}-e \to C_3^*$. We will show, in order, that $e$ cannot be a type of edge listed in the lemma.  

First, suppose $e$ is a loop. The signed graph $C_3^*$ has $g_{_{01}}(C_3^*) = 3$, and so $e$ must be a negative loop. Since $C_3^*$ has negative loops at each vertex, 
$\hat{G}-e \to C_3^*$ if and only if $\hat{G} \to C_3^*$, a contradiction.

Next, suppose edge $f$ is parallel to $e$. The signed graph $C_3^*$ has $g_{_{10}}(C_3^*) = 4$, which means that $e$ and $f$ must have the same sign. But since edges $e$ and $f$ have the same endpoints and sign, $\hat{G}-e \to C_3^*$ if and only if $\hat{G} \to C_3^*$, a contradiction.  

Finally, suppose that $e$ is a cut-edge with ends $u$ and $v$. Let $\hat{G}_u$ and $\hat{G}_v$ be the components of $\hat{G}-e$ containing $u$ and $v$ respectively. Since $\hat{G}$ is $C_3^*$-critical, there exist homomorphisms $\psi_1: \hat{G}_u \to C_3^*$ and $\psi_2: \hat{G}_v \to C_3^*$. By the vertex-transitivity of $C_3^*$, we may assume $\psi_1(u) = \psi_2(v)$ if $e$ is negative and $\psi_1(u) \neq \psi_2(v)$ if $e$ is positive. But, by Observation~\ref{obs:PathExtension}, $\psi_1 \cup \psi_2: \hat{G} \to C_3^*$, a contradiction.
\end{proof}

\begin{lemma}\label{lem: forbidden configs}
No $C_3^*$-critical signed graph contains a vertex of the following type: $1$-vertex, $2_1$-vertex, $4_4$-vertex, or $5_5$-vertex.
\end{lemma}
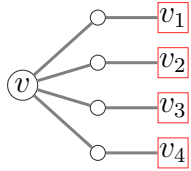
\begin{figure}[ht]
		\centering
		\begin{tikzpicture}[>=latex,
	roundnode/.style={circle, draw=black!80, inner sep=0mm, minimum size=4mm},
	squarednode/.style={rectangle, draw=red!80, inner sep=0mm, minimum size=4mm},
	dotnode/.style={circle, draw=black!80, inner sep=0mm, minimum size=2mm},
	scale=1] 
		\node [roundnode] (v) at (3,1){$v$};
		\node [dotnode](y) at (4,0.7){};
		\node [dotnode](x) at (4,1.3){};
		\node [squarednode](i) at (5,0.7){$v_3$};
		\node [squarednode](j) at (5,1.3){$v_2$};
		
		\node [dotnode](y2) at (4,0.1){};
		\node [dotnode](x2) at (4,1.9){};
		\node [squarednode](i2) at (5,0.1){$v_4$};
		\node [squarednode](j2) at (5,1.9){$v_1$};
		
		\draw [line width =0.4mm, gray, -] (v)--(x)--(j);
		\draw [line width =0.4mm, gray, -] (v)--(y)--(i);
		\draw [line width =0.4mm, gray, -] (v)--(x2)--(j2);
		\draw [line width =0.4mm, gray, -] (v)--(y2)--(i2);
		\end{tikzpicture}
		\caption{A $4_4$-vertex $v$ and its distance-two neighbors.}
		\label{fig:4_4}
\end{figure}
\begin{proof}
Let $\hat{G}$ be a $C_3^*$-critical signed graph.
Suppose to the contrary that there is a vertex $v$ in $\hat{G}$ of a type listed in the lemma. By Lemma~\ref{lem: forbidden edge}, $v$ cannot be a $1$-vertex.

If $v$ is a $2_1$-vertex, let $\hat{G}'$ be the signed graph obtained from $\hat{G}$ by deleting $v$ and its $2$-neighbor. By criticality, there is a homomorphism $\varphi: \hat{G}' \to C_3^*$. But by Observation~\ref{obs:PathExtension}, $\varphi$ can be extended to a homomorphism of $\hat{G}$ to $C_3^*$. This contradicts that $\hat{G}$ is $C_3^*$-critical.

Suppose $v$ is a $4_4$-vertex. Let $v_1,v_2,v_3,v_4$ be the distance-two neighbors of $v$, see Figure~\ref{fig:4_4}. Let $\hat{H}$ be the signed graph obtained from $\hat{G}$ by deleting $v$ and its $2$-neighbors. Since $\hat{H}$ is a proper subgraph of a $C_3^*$-critical signed graph, there is a homomorphism $\varphi$ of $\hat{H}$ to $C^*_3$. Assume that $\varphi$ is edge-sign preserving under the signature $\sigma$. By possibly switching at $v$, we may assume that among four $vv_i$-paths at most two of them are negative, say $vv_1$-path and $vv_2$-path if there exists two. Let $\varphi(v)\in V(C^*_3)\setminus \{\varphi(v_1), \varphi(v_2)\}$. By Observation~\ref{obs:PathExtension}, such a mapping can be extended to those $2$-neighbors of $v$, a contradiction. 

The case when $v$ is a $5_5$-vertex is similar to the $4_4$-vertex case, where again we may assume that there are at most two negative paths, and we omit the proof.
\end{proof}

\begin{lemma}\label{lem: no 32 vertex}
No $C_3^*$-critical signed graph contains a $3_2$-vertex. 
\end{lemma}

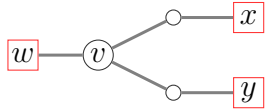
\begin{figure}[ht]
		\centering
		\begin{tikzpicture}[>=latex,
	roundnode/.style={circle, draw=black!80, inner sep=0mm, minimum size=4mm},
	squarednode/.style={rectangle, draw=red!80, inner sep=0mm, minimum size=4mm},
	dotnode/.style={circle, draw=black!80, inner sep=0mm, minimum size=2mm},
	scale=1] 
		\node [squarednode] (u) at (2,1){$w$};
		\node [roundnode] (v) at (3,1){$v$};
		\node [dotnode](y) at (4,0.5){};
		\node [dotnode](x) at (4,1.5){};
		\node [squarednode](i) at (5,0.5){$y$};
		\node [squarednode](j) at (5,1.5){$x$};
		\draw [line width =0.4mm, gray, -] (u)--(v);
		\draw [line width =0.4mm, gray, -] (v)--(x)--(j);
		\draw [line width =0.4mm, gray, -] (v)--(y)--(i);
		\end{tikzpicture}
		\caption{A $3_2$-vertex $v$ and surrounding graph.}
		\label{fig:3_2}
\end{figure}

\begin{proof}
Let $\hat{G}$ be a $C_3^*$-critical signed graph. Suppose to the contrary that there is a $3_2$-vertex $v$ in $\hat{G}$. Let $x$ and $y$ be its distance-two neighbors, and $w$ be the remaining neighbor of $v$. See Figure~\ref{fig:3_2}.

Let $\hat{H}$ be the signed graph obtained from $\hat{G}$ by deleting the vertex $v$ and its two $2$-neighbors. 
By the criticality of $\hat{G}$, there is a homomorphism $\psi: \hat{H} \rightarrow C_3^*$. We claim that $\psi$ can be extended to a homomorphism of $\hat{G}$ to $C_3^*$. Assume that $\sigma$ of $\hat{G}$ is the signature under which $\hat{H}$ admits an edge-sign preserving homomorphism to $C_3^*$. By possibly switching at $v$, we may assume that at most one of the three paths $vw$-path, $vx$-path, and $vy$-path is negative under $\sigma$. Moreover, we may assume each of the three paths has at most one negative edge. If none of the three paths are negative, let $\psi(v) \in V(C_3^*) \setminus \psi(w)$. If $vw$-path is the sole negative path, let $\psi(v) = \psi(w)$. If, without loss of generality, $vx$-path is negative, then let $\psi(v) \in V(C_3^*) \setminus \{\psi(x), \psi(w)\}$. In each case, $\psi$ can be extended to a homomorphism from $G$ to $C_3^*$ by Observation~\ref{obs:PathExtension}.
\end{proof}

\begin{lemma}\label{lem:triangle}
Let $\hat{T}$ be a signed triangle with vertex set $\{v_1,v_2,v_3\}$ and contains at most one negative edge. Let $S_i\subseteq V(C^*_3)$ for $i\in \{1,2,3\}$. There exists an edge-sign preserving homomorphism $\varphi: V(\hat{T}) \to V(C^*_3)$ such that $\varphi(v_i)\in S_i$ for each $i\in \{1,2,3\}$ if one of the following conditions are satisfied: 
\begin{enumerate}[(1)]
\item\label{case:223} $|S_1|\geq 2, |S_2|\geq 2,$ and $|S_3|=3$; 
\item\label{case:133} $|S_1|\geq 1, |S_2|=3,$ and $|S_3|=3$;
\item\label{case:123} $|S_1|\geq 1, |S_2|\geq 2, |S_3|=3,$ and $v_1v_2$ is not the negative edge if $\hat{T}$ is negative;
\item\label{case:222} $|S_1| = |S_2| = |S_3| = 2$, and $S_2 \cup S_3 = V(C_3^*)$.
\end{enumerate}
\end{lemma}

\begin{proof}
For Cases~(\ref{case:223}) and (\ref{case:133}), the argument is the same. In both cases, if $\hat{T}$ contains three positive edges, then we can choose $\varphi(v_1)\in S_1$, $\varphi(v_2)\in S_2\setminus\{\varphi(v_1)\}$ and $\varphi(v_3)\in S_3\setminus \{\varphi(v_1), \varphi(v_2)\}$ in this order. If $\hat{T}$ contains one negative edge $v_iv_j$, then we choose $\varphi(v_i)=\varphi(v_j)\in S_i\cap S_j$ and then for $k\in \{1,2,3\}\setminus \{i,j\}$, $\varphi(v_k)\in S_k\setminus \{\varphi(v_i)\}$. Just be careful in Case (\ref{case:133}), if the only negative edge is $v_2v_3$, then we can always choose $\varphi(v_2)=\varphi(v_3)\in S_2\cap S_3\setminus S_1$ to guarantee that $S_1\setminus \{\varphi(v_2)\}$ is not empty.

For Case (\ref{case:123}), if $\hat{T}$ is positive, then the same argument above works. If $v_2v_3$ is the only negative edge, then we choose $\varphi(v_2)=\varphi(v_3)\in S_2\cap S_3\setminus S_1$ and choose $\varphi(v_1)\in S_1$; if $v_1v_3$ is the only negative edge, then we choose $\varphi(v_1)=\varphi(v_3)\in S_1$ and set $\varphi(v_2)\in S_2\setminus S_1$.

For Case (\ref{case:222}), if $\hat{T}$ is negative then the argument for Case (\ref{case:223}) works. Otherwise $\hat{T}$ is positive and we may choose $\varphi(v_1) \in S_1$. Because $S_2 \cup S_3 = V(C_3^*)$, it follows that $S_2 \setminus \{\varphi(v_1)\} \neq S_3 \setminus \{\varphi(v_1)\}$. Hence we may choose $\varphi(v_2) \in S_2 \setminus \{\varphi(v_1)\}$ and $\varphi(v_3) \in S_3 \setminus \{\varphi(v_1)\}$ so that $\varphi(v_2) \neq \varphi(v_3)$.
\end{proof}

\begin{lemma}\label{lem: C_3* crit forbidden triangles with 31}
No signed triangle $\hat{T}$ of the following type is contained in a $C_3^*$-critical signed graph:
\begin{enumerate}[(i)]
\item\label{caseC: 31-31-53} two $3_1$-vertices and a $5_3$-vertex;
\item\label{caseC: 31-42-42} a $3_1$-vertex and two $4_{2}$-vertices;
\item\label{caseC: 31-30-42} a $3_1$-vertex, a $3$-vertex, and a $4_{2}$-vertex.
\end{enumerate}
\end{lemma}

\begin{figure}[ht]
    \centering
	\begin{tikzpicture}[>=latex,
	roundnode/.style={circle, draw=black!80, inner sep=0mm, minimum size=4mm},
	squarednode/.style={rectangle, draw=red!80, inner sep=0mm, minimum size=4mm},
	dotnode/.style={circle, draw=black!80, inner sep=0mm, minimum size=2mm},
	scale=1] 
	
	\begin{scope}[xshift=-0.5cm, yshift = 0cm]
	\node [squarednode] (x) at (5, 1.6){\footnotesize $x_3$};
	\node [squarednode] (x3) at (5, 1){\footnotesize $x_4$};
	\node [squarednode] (x4) at (5, 0.4){\footnotesize $x_5$};
	\node [dotnode] (v32) at (4, 1.6){};
	\node [dotnode] (v33) at (4, 1){};
	\node [dotnode] (v34) at (4, 0.4){};
	\node [roundnode] (v) at (3, 1){\footnotesize $v_3$};
	\node [roundnode] (v1) at (2, 1.5){\footnotesize $v_1$};
	\node [dotnode] (v11) at (1, 1.5){};
	\node [squarednode] (v1') at (0, 1.5){\footnotesize $x_1$};
	\node [roundnode] (v2) at (2, 0.5){\footnotesize $v_2$};
	\node [dotnode] (v22) at (1, 0.5){};
	\node [squarednode] (v2') at (0, 0.5){\footnotesize $x_2$};
	\draw [line width =0.4mm, gray, -] (v2')--(v22)--(v2)--(v)--(v32)--(x); 
	\draw [line width =0.4mm, gray, -] (v1')--(v11)--(v1)--(v);
	\draw [line width =0.4mm, gray, -] (v1)--(v2);
	\draw [line width =0.4mm, gray, -] (x3)--(v33)--(v)--(v34)--(x4);
	\node at (2.5,-0.2) {\footnotesize (\ref{caseC: 31-31-53})};
	\node at (2.33,1) {\footnotesize $\hat{T}$};
	\end{scope}
	
	\begin{scope}[xshift=6.5cm, yshift = 0cm]
	\node [squarednode] (x) at (5, 1.5){\footnotesize $x_{3}$};
	\node [squarednode] (x1) at (5, 2.1){\footnotesize $x_{2}$};
	\node [dotnode] (v21) at (4, 1.5){};
	\node [dotnode] (v22) at (4, 2.1){};
	\node [squarednode] (y) at (5, 0.5){\footnotesize $x_{4}$};
	\node [squarednode] (y1) at (5, -0.1){\footnotesize $x_{5}$};
	\node [dotnode] (v31) at (4, 0.5){};
	\node [dotnode] (v32) at (4, -0.1){};
	\node [roundnode] (v) at (3, 1.5){\footnotesize $v_{2}$};
	\node [roundnode] (v1) at (2, 1){\footnotesize $v_{1}$};
	\node [dotnode] (v11) at (1, 1){};
	\node [squarednode] (v1') at (0, 1){\footnotesize $x_{1}$};
	\node [roundnode] (v2) at (3, 0.5){\footnotesize $v_{3}$};
	\draw [line width =0.4mm, gray, -] (v2)--(v)--(v21)--(x); 
	\draw [line width =0.4mm, gray, -] (v)--(v22)--(x1);
	\draw [line width =0.4mm, gray, -] (v2)--(v32)--(y1);
	\draw [line width =0.4mm, gray, -] (v1')--(v11)--(v1)--(v);
	\draw [line width =0.4mm, gray, -] (v1)--(v2)--(v31)--(y);
	\node at (2.5,-0.2) {\footnotesize (\ref{caseC: 31-42-42})};
	\node at (2.68,1) {\footnotesize $\hat{T}$};
	\end{scope}
	
	\begin{scope}[xshift=3.25cm, yshift = -2.2cm]
	\node [squarednode] (x) at (4, 1.5){\footnotesize $x_2$};
	\node [squarednode] (y) at (5, 0.5){\footnotesize $x_3$};
	\node [squarednode] (y1) at (5, -0.1){\footnotesize $x_4$};
	\node [dotnode] (v31) at (4, 0.5){};
	\node [dotnode] (v32) at (4, -0.1){};
	\node [roundnode] (v) at (3, 1.5){\footnotesize $v_2$};
	\node [roundnode] (v1) at (2, 1){\footnotesize $v_1$};
	\node [dotnode] (v11) at (1, 1){};
	\node [squarednode] (v1') at (0, 1){\footnotesize $x_1$};
	\node [roundnode] (v2) at (3, 0.5){\footnotesize $v_3$};
	\draw [line width =0.4mm, gray, -] (v2)--(v)--(x); 
	\draw [line width =0.4mm, gray, -] (v1')--(v11)--(v1)--(v);
	\draw [line width =0.4mm, gray, -] (v1)--(v2)--(v31)--(y);
	\draw [line width =0.4mm, gray, -] (v2)--(v32)--(y1);
	\node at (2.5,-0.4) {\footnotesize (\ref{caseC: 31-30-42})};
	\node at (2.68,1) {\footnotesize $\hat{T}$};
	\end{scope}
	
	\end{tikzpicture}
	\caption{The three cases in Lemma \ref{lem: C_3* crit forbidden triangles with 31}}
	\label{Cfig:3_1-3_1-4}
\end{figure}
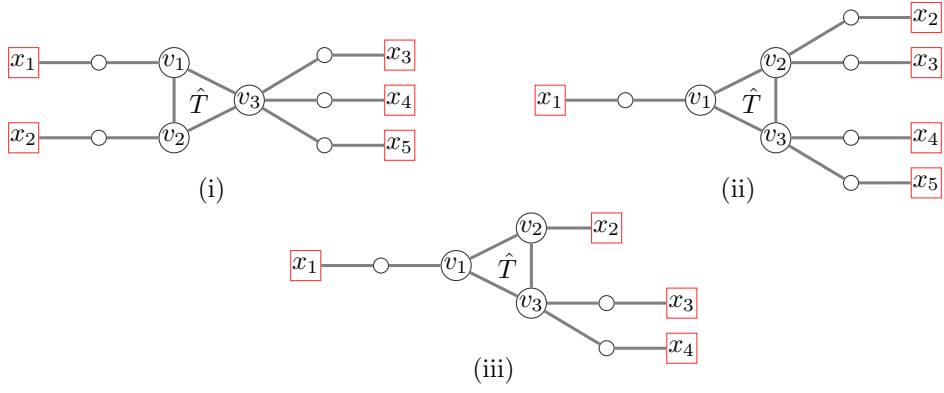

\begin{proof}
Let $\hat{G}$ be a $C_3^*$-critical signed graph. Let the vertices at distance at most two from $\hat{T}$ be labeled as in Figure~\ref{Cfig:3_1-3_1-4}. We proceed by cases. In each case, suppose for contradiction that the described signed triangle $\hat{T}$ does exist in $\hat{G}$. Let $\mathcal{P}$ denote the set of paths in $\hat{G}\setminus E(\hat{T})$ drawn in Figure \ref{Cfig:3_1-3_1-4} that join $v_i$ and $x_j$ for some $i,j$. Let $N$ be the internal vertices of the paths of $\mathcal{P}$. Let $\hat{H}$ denote the signed graph obtained from $\hat{G}$ by deleting $V(T)\cup N$. By the criticality of $\hat{G}$, there is a homomorphism $\psi: \hat{H} \rightarrow C_3^*$. Let $\sigma$ be a signature of $\hat{G}$ such that $(H, \sigma|_H)$ admits an edge-sign preserving homomorphism to $C^*_3$. By possibly switching at some subset of $V(\hat{T})$, we may assume that there is at most one negative edge in $\hat{T}$ with respect to $\sigma$. 

\medskip
\noindent
{\bf Cases (\ref{caseC: 31-31-53}) and (\ref{caseC: 31-42-42}).} By possibly switching on the set $V(T)$, we may assume at most two of the five paths in $\mathcal{P}$ are negative under $\sigma$. Further, by possibly switching on the vertices in $N$, we may assume each path of $\mathcal{P}$ has at most one negative edge. We proceed in two sub-cases. 

First, suppose that there is at most one negative path in $\mathcal{P}$, and denote the end point of that path in $\hat{T}$ by $v_k$ and the other endpoint by $x_k$. We define $S_k=V(C_3^*)\setminus \psi(x_k)$ and $S_i=S_j=V(C^*_3)$ for $v_i,v_j\in V(\hat{T})\setminus \{v_k\}$. Noting $|S_k|=2, |S_i|=|S_j|=3$, by Lemma~\ref{lem:triangle}~(\ref{case:223}), we can choose $\psi(v_\ell)\in S_\ell$ for $\ell\in \{1,2,3\}$ such that $\psi: V(\hat{G})\setminus N\to V(C_3^*)$ is an edge-sign preserving homomorphism of $\hat{G}-N$ to $C^*_3$. Then by Observation~\ref{obs:PathExtension}, we can extend this homomorphism to the remaining vertices of $\hat{G}$. 

Second, if exactly two paths are negative, then let $v_a, v_b$ be the ends of those paths on $T$ and let $x_i, x_j$ be the other ends of these two paths, respectively. If $v_a \neq v_b$, then set $S_a=V(C^*_3)\setminus \{\psi(x_i)\}$, $S_b=V(C^*_3)\setminus \{\psi(x_j)\}$ and $S_k=V(C^*_3)$ where $k\in [3]\setminus \{a,b\}$. Since $|S_a|=|S_b|=2$ and $|S_k|=3$, by Lemma~\ref{lem:triangle}~(\ref{case:223}), we can choose $\psi(v_\ell)\in S_\ell$ for $\ell\in \{1,2,3\}$ such that $\psi: V(\hat{G})\setminus N\to V(C_3^*)$ is an edge-sign preserving homomorphism of $\hat{G}-N$ to $C^*_3$. Then by Observation~\ref{obs:PathExtension}, we can extend this homomorphism to the remaining vertices of $\hat{G}$. 
Similarly, if $v_a = v_b$, then we set $S_a=V(C^*_3)\setminus \{\psi(x_i), \psi(x_j)\}$ and set $S_\alpha=S_\beta=V(C^*_3)$ where $\alpha, \beta\in [3]\setminus \{a\}$. Again, by Lemma~\ref{lem:triangle}~(\ref{case:133}) and Observation~\ref{obs:PathExtension} we can extend this homomorphism to the remaining vertices of $\hat{G}$. In both cases, it is a contradiction as $\hat{G} \rightarrow C_3^*$.

\medskip
\noindent
{\bf Case (\ref{caseC: 31-30-42}).} We may assume $\sigma$ satisfies the following: If $\hat{T}$ has a negative edge, then it is $v_1v_3$; At most two of the four paths in $\mathcal{P}$ are negative; If exactly two such paths are negative, then $v_2x_2$ is negative. The first is accomplished by switching on some subset of $V(\hat{T})$, and the last two by possibly switching on the set $V(\hat{T})$. As before, we assume each path of $\mathcal{P}$ has at most one negative edge.

No matter the sign of $v_2x_2$, there is at most one negative path in $\mathcal{P}\setminus \{v_2x_2\}$. Let $v_k \in \{v_1,v_3\}$ so that $v_k$ is an endpoint of the negative path (and $x_k$ is the other end) if it exists. We proceed in two sub-cases based on the sign of $v_2x_2$.

If $v_2x_2$ is positive, then we set $S_2=V(C^*_3)\setminus \{\psi(x_2)\}, S_k=V(C^*_3)\setminus \{\psi(x_k)\}$ and $S_\ell=V(C^*_3)$ where $\ell\in [3]\setminus \{2, k\}$. By a similar argument as above and by Lemma~\ref{lem:triangle}~(\ref{case:223}), we are done. 

If $v_2x_2$ is negative, then we set $S_2=\{\psi(x_2)\}, S_k=V(C^*_3)\setminus \{\psi(x_k)\}$ and $S_\ell=V(C^*_3)$ where $\ell\in [3]\setminus \{2, k\}$. By a similar argument as above and by Lemma~\ref{lem:triangle}~(\ref{case:123}), we are done. 
\end{proof}

\section{Density of $C_3^*$-critical signed graphs}\label{sec:MainProof}

The \emph{potential} of a graph $G$ is defined as $$\rho(G) = 3v(G)-2e(G).$$ The potential of a signed graph is the potential of its underlying graph. We first give the potential of some simple graphs.
\begin{observation}\label{obs: potential of small complete graphs}
We have $\rho(K_1) = 3$, $\rho(K_2) = 4$, $\rho(K_3) = 3$, and $\rho(K_4) = 0$. 
\end{observation}

In this section, we shall prove the following alternative formulation of Theorem~\ref{thm:main result}.

\begin{theorem}\label{thm:main result alt}
If $\hat{G}$ is $C^*_3$-critical, then $\rho(G) \leq 1$. 
\end{theorem}

To prove it, we assume to the contrary that there exists a $C^*_3$-critical signed graph $\hat{G}$ with $\rho(G) \geq 2$ such that among all such counterexamples to Theorem~\ref{thm:main result alt} it has the minimum number of vertices. This means any $C^*_3$-critical signed graph $\hat{G}'$ with $|V(G')|<|V(G)|$ satisfies that $\rho(G')\leq 1$. We fix the minimum counterexample $\hat{G}$ for the rest of this section and finally arrive at a contradiction to prove Theorem~\ref{thm:main result alt}. Initially, we will develop several structural properties of $\hat{G}$, after which we apply a discharging argument to force a contradiction. 

Given a (signed) graph $\hat{H}$, let $P_2(\hat{H})$ denote a graph obtained from (the underlying graph of) $\hat{H}$ by adding a new $2$-vertex incident to two new edges whose other ends are distinct vertices in $\hat{H}$. By a slight abuse of notation, we sometimes treat $P_2(\hat{H})$ as a signed graph.

\begin{lemma}\label{lem: main lemma}
Let $\hat{H}$ be a subgraph of $\hat{G}$. Then 
\begin{enumerate}[(i)]
    \item\label{case: G=H} $\rho({H}) \geq 2$, if $\hat{G} = \hat{H}$;
    \item\label{case: G=P2H} $\rho({H}) \geq 3$, if $\hat{G} = P_2(\hat{H})$;
    \item\label{case: K1 K3} $\rho({H}) = 3$, if $H = K_1$ or $K_3$;
    \item\label{case: p geq4} $\rho({H}) \geq 4$, otherwise.
\end{enumerate}
\end{lemma}

\begin{proof}
It is straightforward to verify (\ref{case: G=H}), (\ref{case: G=P2H}), and (\ref{case: K1 K3}). Indeed, if $\hat{H} = \hat{G}$ then the lemma is satisfied by our assumption that $\rho(G) \geq 2$; if $P_2(\hat{H}) = \hat{G}$, then $\rho(H) = \rho(G) -3+4 \geq 3$; if $H = K_1$ or $K_3$, then the lemma is satisfied by Observation~\ref{obs: potential of small complete graphs}. 

\smallskip
Suppose for contradiction that (\ref{case: p geq4}) is false. Let $\hat{H}\subseteq \hat{G}$ be a subgraph which does not satisfy (\ref{case: p geq4}), chosen so that among all such subgraphs $v(\hat{H}) + e(\hat{H})$ is maximum. Note that $\hat{H}\neq \hat{G}$, $P_2(\hat{H})\neq \hat{G}$, $\hat{H}\not\in \{K_1, K_3\}$, and $\rho(\hat{H}) \leq 3$. We first claim that $\hat{H}$ is an induced subgraph of $\hat{G}$. Otherwise, assume that $e \not\in E(\hat{H})$ is an edge connecting two vertices of $\hat{H}$. Note that $\rho(\hat{H}+e) = \rho(\hat{H}) - 2 \leq 1< \rho(\hat{H})$, and $v(\hat{H}+e)+e(\hat{H}+e)>v(\hat{H})+e(\hat{H})$. By (\ref{case: G=H}), (\ref{case: G=P2H}), and (\ref{case: K1 K3}), it follows that $\hat{H}+e\neq \hat{G}$, $P_2(\hat{H}+e)\neq \hat{G}$, $\hat{H}+e\not\in \{K_1, K_3\}$. Thus the existence of $\hat{H}+e$ contradicts the maximality of $\hat{H}$. Therefore, $\hat{H}$ is an induced subgraph.

By Observation~\ref{obs: potential of small complete graphs}, we know that $v(\hat{H}) \geq 4$. Since $\hat{H}$ is a proper subgraph of $\hat{G}$ and $\hat{G}$ is $C_3^*$-critical, there exists a homomorphism $\psi: \hat{H} \rightarrow C_3^*$. We may assume that $\sigma$ is a signature of $\hat{G}$ such that $\psi$ is an edge-sign preserving homomorphism of $(H, \sigma|_{H})$ to $C_3^*$. We build a signed graph $\hat{G}_1$ from $\hat{G}$ by identifying any $u,v \in V(\hat{H})$ whenever $\psi(u) = \psi(v)$ and deleting resulting parallel edges of the same sign so that only one representative remains.

We proceed with four observations about $\hat{G}_1$.
First, $\hat{G}_1$ has no positive loops because $\psi(u) = \psi(v)$ only if $uv$ is not a positive edge in $\hat{G}$.
Second, $v(\hat{G}_1) + e(\hat{G}_1) < v(\hat{G}) + e(\hat{G})$ as $v(\hat{H}) \geq 4> v(C^*_3)$ which means at least two vertices were identified while forming $\hat{G}_1$. Third, $\hat{G}_1 \not\rightarrow C_3^*$ because otherwise $\hat{G} \rightarrow C_3^*$ by the transitivity of homomorphisms, which would contradict that $\hat{G}$ is $C_3^*$-critical. Fourth and finally, $\hat{G}_1$ has no digon. Suppose to the contrary that there is a digon in $\hat{G}_1$. This means that in $\hat{G}$ there is a negative path $\hat{P}$ of length $2$ (with respect to the signature $\sigma$) so that the ends of $\hat{P}$ are in $V(\hat{H})$ while the internal vertex of $\hat{P}$ is in $V(\hat{G} - \hat{H})$. Then $\rho(\hat{P} + \hat{H})=\rho(\hat{H}) +3 -4 \leq 2$. Note that $\hat{P} + \hat{H}$ is also a subgraph of $\hat{G}$ but has more vertices plus edges than $\hat{H}$ and moreover, $\rho(\hat{P} + \hat{H}) \leq 2< \rho(\hat{H})$. 
By the choice of $\hat{H}$ and claim~(\ref{case: G=H}), it must be that $\hat{G}=\hat{P}+\hat{H}$, contradicting that $\hat{G} \neq P_2(\hat{H})$. Thus $\hat{G}_1$ has no digon.

By the first, third, and fourth observations above, $\hat{G}_1$ contains a $C_3^*$-critical subgraph $\hat{G}_2$. Let $X \subseteq V(\hat{G}_1)$ be the identified vertices of $\hat{H}$ (including trivial identification if $\psi^{-1}(u)$ is a singleton set for some $u\in V(C^*_3)$). 
First, observe that $X \cap V(\hat{G}_2) \neq \emptyset$ because otherwise $\hat{G}_2 \subsetneq \hat{G}$ but both of them are $C_3^*$-critical. Also, $V(\hat{G}_2)\setminus X \neq \emptyset$ because otherwise $\hat{G}_2$ is a subgraph of $C^*_3$ which would mean $\hat{G}_2 \rightarrow C_3^*$, a contradiction. Since $v(\hat{G}_2) + e(\hat{G}_2)\leq v(\hat{G}_1) + e(\hat{G}_1) < v(\hat{G}) + e(\hat{G})$, by the choice of the minimum counterexample $\hat{G}$ (to Theorem~\ref{thm:main result alt}), we know that
$\rho(\hat{G}_2) \leq 1$.

We now construct a signed graph $\hat{G}_3$ from the disjoint union of $\hat{G}_2-X$ and $\hat{H}$ by adding the following edges. For each vertex $v \in V(\hat{G}_2) \setminus X$ and each vertex $u \in X$, if $vu \in E(\hat{G}_2)$, then choose a representative edge $vw \in E(\hat{G})$ for some $w \in \psi^{-1}(u)$ to be included in $E(\hat{G}_3)$. In this way, $\hat{G}_3 \subseteq \hat{G}$, and because $V(\hat{G}_2)\setminus X \neq \emptyset$, $\hat{H} \subsetneq \hat{G}_3$.

Now we consider the vertices and edges in $\hat{G_3}$. It is straightforward to see that $V(\hat{G}_3) = V(\hat{G}_2) \cup V(\hat{H}) \setminus V(X)$, and so
\begin{equation}\label{eqn:vertices of G3}
    v(\hat{G}_3) = v(\hat{G}_2) + v(\hat{H}) - |X|.
\end{equation}
An edge of $\hat{G}_3$ is of one of three types: (a) edges in $\hat{G}_2 \setminus X$; (b) edges in $\hat{H}$; and (c) edges with one end in $\hat{G}_2 \setminus X$, and the other end in $\hat{H}$.
Similarly, edges in $\hat{G_2}$ have one of three types: (a) edges in $\hat{G}_2 \setminus X$; (d) edges with both ends in $X$; and (e) edges with one end in $\hat{G}_2 \setminus X$, and the other end in $X$. But, by construction, there is a one-to-one correspondence between edges of type (c) and (e). Therefore,
\begin{equation}\label{eqn:edges of G3}
    e(\hat{G}_3) = e(\hat{G}_2) + e(\hat{H}) - e(G_2[X]).
\end{equation}
Hence, noting $\rho(\hat{G}_2) \leq 1$, $\rho(\hat{H}) \leq 3$, and $\rho(G_2[X]) \geq 3$ (by Observation \ref{obs: potential of small complete graphs}), by Equations~(\ref{eqn:vertices of G3}) and (\ref{eqn:edges of G3}), it follows that $$\rho(\hat{G}_3) = \rho(\hat{G}_2) + \rho(\hat{H}) - \rho(G_2[X])\leq 1+3-3\leq 1.$$ 
Since $\hat{H} \subsetneq \hat{G}_3 \subseteq \hat{G}$ and $\rho(\hat{G}_3)<\rho(\hat{H})$, by claims~(\ref{case: G=H}), (\ref{case: G=P2H}), and (\ref{case: K1 K3}), $\hat{G}_3$ is a larger subgraph of $\hat{G}$ than $\hat{H}$ which doesn't satisfy claim~(\ref{case: p geq4}). This contradicts our choice of $\hat{H}$ and completes the proof.
\end{proof}

The following observation will aid in the proof of Lemma~\ref{lem: forbidden theta subgraphs}.

\begin{observation}\label{obs:cycles for forbidden theta subgraph lemma}
Let $C$ be a cycle of a $C_3^*$-critical signed graph with a vertex of degree~$2$. \begin{enumerate}[(i)]
    \item\label{case:4cycle} If $C$ is a $4$-cycle, then $C$ is negative. If, additionally, $C$ has a chord then the two triangles formed by this chord have different signs.
    \item\label{case:3cycle} If $C$ is a $3$-cycle, then $C$ is positive.
\end{enumerate}
\end{observation}
To justify the observation, note that if $C$ does not have the prescribed sign, then the subgraph formed by deleting the 2-vertex maps to $C_3^*$ if and only if $\hat{H}$ does.

\smallskip
Let $\Theta_1, \Theta_2, \Theta_3,$ and $X$ be the graphs depicted in Figure~\ref{fig:Theta1} \ref{fig:Theta2}, \ref{fig:Theta3}, and \ref{fig:X}, respectively.

\begin{figure}[ht]
    \centering 
    \begin{subfigure}[t]{.23\textwidth}
		\centering
    \begin{tikzpicture}[>=latex, 
        roundnode/.style={circle, draw=black!80, inner sep=0mm, minimum size=4mm},
	    scale=.26] 
	\begin{scope}[xshift=0cm, yshift = 0cm]
    \foreach \i in {1,2,3,4}{
        \draw[rotate=90*(\i)-90] (0,4) node[roundnode] (x_{\i}){\scriptsize ${v_{\i}}$};}
	
    \foreach \i/\j in {1/2,2/3,3/4,4/1,2/4}{
        \draw[line width=0.4mm, gray] (x_{\i}) -- (x_{\j});}
    \end{scope}
    \end{tikzpicture}
    \caption{$\Theta_1$}
		\label{fig:Theta1}     
   \end{subfigure}
   \begin{subfigure}[t]{.23\textwidth}
		\centering
    \begin{tikzpicture}[>=latex, 
        roundnode/.style={circle, draw=black!80, inner sep=0mm, minimum size=4mm},
	    scale=.26] 
    \begin{scope}[xshift=10cm, yshift = 0cm]
    \node[roundnode] (x_{5}){\scriptsize ${v_5}$};
    \foreach \i in {1,2,3,4}{
        \draw[rotate=90*(\i)-90] (0,4) node[roundnode] (x_{\i}){\scriptsize ${v_{\i}}$};}
	
    \foreach \i/\j in {1/2,2/3,3/4,4/1,2/5,5/4}{
        \draw[line width=0.4mm, gray] (x_{\i}) -- (x_{\j});}
    \end{scope}
    \end{tikzpicture}
    \caption{$\Theta_2$}
		\label{fig:Theta2}     
   \end{subfigure}
   \begin{subfigure}[t]{.23\textwidth}
		\centering
    \begin{tikzpicture}[>=latex, 
        roundnode/.style={circle, draw=black!80, inner sep=0mm, minimum size=4mm},
	    scale=.26] 
    \begin{scope}[xshift=20cm, yshift = 0cm]

    \node[roundnode] (x_{1}) at (0,4) {\scriptsize ${v_1}$};
    \node[roundnode] (x_{2}) at (-3.5,0) {\scriptsize ${v_2}$};
    \node[roundnode] (x_{3}) at (-3.5,-4) {\scriptsize ${v_3}$};
    \node[roundnode] (x_{4}) at (3.5,-4) {\scriptsize ${v_4}$};
    \node[roundnode] (x_{5}) at (3.5,0) {\scriptsize ${v_5}$};
	
    \foreach \i/\j in {1/2,2/3,3/4,4/5,5/1,2/5}{
        \draw[line width=0.4mm, gray] (x_{\i}) -- (x_{\j});}
    \end{scope}
    \end{tikzpicture}
    \caption{$\Theta_3$}
		\label{fig:Theta3}     
   \end{subfigure}
   \begin{subfigure}[t]{.23\textwidth}
		\centering
    \begin{tikzpicture}[>=latex, 
        roundnode/.style={circle, draw=black!80, inner sep=0mm, minimum size=4mm},
	    scale=.26] 
    \begin{scope}[xshift=30cm, yshift = 0cm]
    \node[roundnode] (x_{1}) at (-3.5,4) {\scriptsize ${v_1}$};
    \node[roundnode] (x_{2}) at (3.5,4) {\scriptsize ${v_2}$};
    \node[roundnode] (x_{3}) at (0,0) {\scriptsize ${v_3}$};
    \node[roundnode] (x_{4}) at (-3.5,-4) {\scriptsize ${v_4}$};
    \node[roundnode] (x_{5}) at (3.5,-4) {\scriptsize ${v_5}$};
	
    \foreach \i/\j in {1/2,2/3,3/1,3/4,4/5,3/5}{
        \draw[line width=0.4mm, gray] (x_{\i}) -- (x_{\j});}
    \end{scope}
    \end{tikzpicture}
    \caption{$X$}
		\label{fig:X}     
   \end{subfigure}
    \caption{Graphs that are not a subgraph of $\hat{G}$.}
    \label{fig:H not isomorphic}
\end{figure}
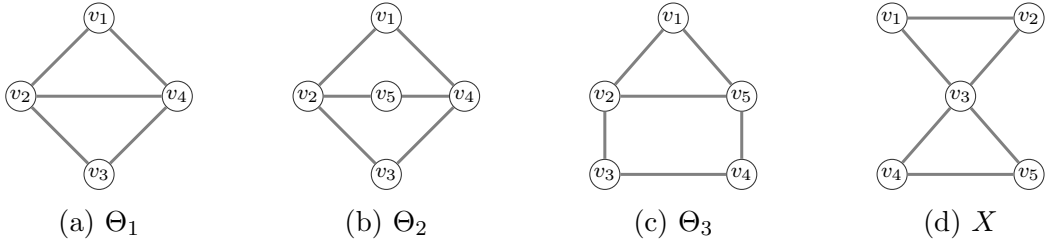

\begin{lemma}\label{lem: forbidden theta subgraphs}
$\hat{G}$ has no subgraph whose underlying graph is isomorphic to any of $\Theta_1, \Theta_2, \Theta_3$, or $X$.
\end{lemma}

\begin{proof}
We can easily compute that $\rho(\Theta_1)=2$ and $\rho(\Theta_2)=\rho(\Theta_3)=\rho(X)=3$. Let $\hat{H}$ be a subgraph of $\hat{G}$ and $H\in \{\Theta_1, \Theta_2, \Theta_3, X\}$. 

\smallskip
Suppose that $H = \Theta_1$. By Lemma~\ref{lem: main lemma}~(\ref{case: G=H}), we conclude that $G=\Theta_1$. By Lemma~\ref{lem:theta}, it is not possible. Suppose that $H \in \{\Theta_2,\Theta_3,X\}$. By Lemma~\ref{lem: main lemma}~(\ref{case: G=H}) and (\ref{case: G=P2H}), we conclude that either $G=H$ or $G=P_2(H)$. If $H \in \{\Theta_2,\Theta_3\}$, then the former case is impossible by Lemma~\ref{lem:theta}. If $H = X$, then the former case would contradict that $\hat{G}$ is $C_3^*$-critical. Indeed, if $X$ could be signed in a way that made it $C_3^*$-critical, then by Observation \ref{obs:cycles for forbidden theta subgraph lemma} (\ref{case:3cycle}), it must have two triangles $v_1v_2v_3$ and $v_3v_4v_5$ both being positive, which is switching equivalent to the signature that all the edges of $\hat{X}$ are positive, and it is easy to see that when signed this way $\hat{X} \to C_3^*$, a contradiction.

\medskip
We are left to show that $P_2(\Theta_2)$, $P_2(\Theta_3)$, and $P_2(X)$ each have no signature which makes them $C_3^*$-critical. We consider them in order. In each case, let $v$ be the new vertex and let the remaining vertices be labeled as in Figure~\ref{fig:H not isomorphic}. Suppose for contradiction that $\sigma$ is a signature of each respective graph so that it is $C_3^*$-critical. 

\begin{itemize}
\item  For $P_2(\Theta_2)$, $v$ must be adjacent to two of $\{v_1,v_3,v_5\}$. Otherwise, the $4$-cycles $v_1v_2v_5v_4$, $v_1v_2v_3v_4$ and $v_5v_2v_3v_4$, each containing a $2$-vertex, would need to be negative by Observation~\ref{obs:cycles for forbidden theta subgraph lemma}~(\ref{case:4cycle}), but it is impossible by the handshake lemma (i.e., the number of negative facial cycles of a signed plane graph is even). Thus, without loss of generality, it must be that $vv_1,vv_3 \in E(P_2(\Theta_2))$. Again by Observation~\ref{obs:cycles for forbidden theta subgraph lemma}~(\ref{case:4cycle}), we need $P_2(\Theta_2)$ to be assigned such that every $4$-cycle with a $2$-vertex is negative and, by Lemma~\ref{lem:SwitchingEquivalent}, it determines a unique (switching-equivalent) signature $\sigma$. Note that up to switching we may assume that $v_1v_4$ and $v_2v_3$ are the only negative edges in $\sigma$. But now it is straightforward to verify that $(P_2(\Theta_2),\sigma) \to C_3^*$, a contradiction.

\item For $P_2(\Theta_3)$, $v$ must be adjacent to one of $v_3$ or $v_4$ by Lemma~\ref{lem: forbidden configs} which forbids a $2_1$-vertex in any $C_3^*$-critical signed graph. By symmetry, say $v$ is adjacent to $v_3$. If $v$ is also adjacent to $v_5$, then the underlying graph is isomorphic to a $P_2(\Theta_2)$, and this case is complete. If $v$ is also adjacent to $v_1$, then the 4-cycles  $v_1v_2v_3v$ and $v_5v_2v_3v_4$ must be negative by Observation~\ref{obs:cycles for forbidden theta subgraph lemma}~(\ref{case:4cycle}). But then by assuming that three edges of $v_1v_2v_3$ are all negative or all positive depending on the sign of the triangle, it is straightforward to see the signed graph maps to $C_3^*$. The cases for $v$ being adjacent to $v_2$ or $v_4$ are omitted because they proceed as in the previous case, where Observation~\ref{obs:cycles for forbidden theta subgraph lemma} reduces the number of signatures to be considered to one or two, and then, in a straightforward way, verifying that a homomorphism to $C_3^*$ does exist.

\item For $P_2(X)$, it must be that $v$ is adjacent to a $2$-vertex in each of the triangles $v_1v_2v_3$ and $v_3v_4v_5$ by Lemma~\ref{lem: forbidden configs} as there are no two $2$-vertices adjacent to each other in any $C_3^*$-critical signed graph.By symmetry, we may assume that $v$ is adjacent to $v_2$ and $v_5$. But this graph is isomorphic to a $P_2(\Theta_3)$ (i.e., $v$ is adjacent to $v_2$ and $v_3$ in $P_2(\Theta_3)$) and, by the previous case, it cannot be the underlying graph of a $C_3^*$-critical signed graph.
\end{itemize}
This completes the proof.
\end{proof}

\begin{corollary}\label{cor:unique triangle}
Every vertex of $\hat{G}$ is in at most one triangle.
\end{corollary}
\begin{proof}
If a vertex is contained in two triangles, then, as a $C_3^*$-critical graph has no parallel edges or digon, those two triangles must either (1) share exactly one vertex, or (2) share exactly one edge. But both contradict Lemma~\ref{lem: forbidden theta subgraphs}.
\end{proof}

\begin{lemma}\label{lem: 31 in triangle}
Let $v$ be a $3_1$-vertex of $\hat{G}$ and $u$ be its $2$-neighbor. Assume that $x$ and $y$ are the $3^+$-neighbors of $v$. Then the path $xvy$ must be in a positive triangle. 
\end{lemma}

\begin{figure}[ht]
		\centering
		\begin{tikzpicture}[>=latex,
	roundnode/.style={circle, draw=black!80, inner sep=0mm, minimum size=4mm},
	squarednode/.style={rectangle, draw=red!80, inner sep=0mm, minimum size=4mm},
	dotnode/.style={circle, draw=black!80, inner sep=0mm, minimum size=2mm},
	scale=1] 
		\node [roundnode] (u) at (2,1){$u$};
		\node [roundnode] (v) at (3,1){$v$};
		\node [squarednode] (w) at (1,1){$w$};
		\node [squarednode](y) at (4,0.5){$y$};
		\node [squarednode](x) at (4,1.5){$x$};
		\draw [line width =0.4mm, gray] (u)--(w);
		\draw [line width =0.4mm, gray, -] (u)--(v);
		\draw [line width =0.4mm, gray, -] (v)--(x);
		\draw [line width =0.4mm, gray, -] (v)--(y);
		\end{tikzpicture}
		\caption{A $3_1$-vertex with its neighbors.}
		\label{8FC3_1-3}
\end{figure}
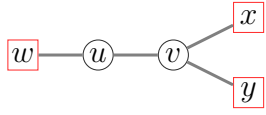

\begin{proof}
Let $w$ be the other neighbor of $u$ which is not $v$. See Figure~\ref{8FC3_1-3}. Suppose, for a contradiction, that either the path $xvy$ is in a negative triangle or $xy \not\in E(\hat{G})$. We consider these two possibilities:

\begin{itemize}

\item Suppose that $xvy$ is a negative triangle. Let $\hat{G}_1 = \hat{G} - \{u,v\}$. It follows from the criticality of $\hat{G}$ that there is a homomorphism $\psi: \hat{G}_1 \rightarrow C_3^*$. Let $\sigma$ be a signature of $\hat{G}$ such that $\psi: (\hat{G}_1, \sigma |_{G_1}) \rightarrow C_3^*$ is an edge-sign preserving homomorphism with respect to $\sigma$. We will arrive at a contradiction by showing that $\psi$ can be extended to $\hat{G}$. By possibly switching on $u$ and $v$, we may assume $\sigma(uv) = \sigma(uw) = +$.

In the edge-sign preserving homomorphism $\psi:\hat{G}_1 \rightarrow C_3^*$, based on the sign of the edge $xy$, to determine $\psi(v)$ we have two cases to consider: (1) If $\sigma(xy) = -$, then $\psi(x) = \psi(y)$. Because $xvy$ is a negative triangle, it must be that $\sigma(xv) = \sigma(yv)$. If both are negative, then set $\psi(v)= \psi(x)$, otherwise choose $\psi(v)\in V(C^*_3)\setminus \{\psi(x)\}$. (2) If $\sigma(xy) = +$, then $\psi(x) \neq \psi(y)$ and $\sigma(xv) \neq \sigma(yv)$. Suppose without loss of generality that $\sigma(xv) = -$ and then we set $\psi(v) = \psi(x)$. Now we determine $\psi(u)$. Since $\sigma(uv) = \sigma(uw) = +$ we may extend $\psi$ by setting $\psi(u)\in V(C^*_3)\setminus \{\psi(v), \psi(w)\}$. Therefore, $\hat{G} \rightarrow C_3^*$, a contradiction.

\item Suppose that $xy \not\in E(\hat{G})$. Let $\hat{G}_1 = \hat{G} - \{u,v\} + xy$ and assign a sign to $xy$ such that $xvy$ is a negative triangle in $\hat{G}+xy$. By the above reasoning, if $\hat{G}_1 \rightarrow C_3^*$, then also $\hat{G} \rightarrow C_3^*$, a contradiction. Thus $\hat{G}_1 \not\rightarrow C_3^*$, and so there exists $\hat{G}_2 \subseteq \hat{G}_1$ that is $C_3^*$-critical. Observe that in constructing $\hat{G}_2$ we do not create any digon. Clearly, $xy \in E(\hat{G}_2)$ because otherwise $\hat{G}_2 \subsetneq \hat{G}$ but both are $C_3^*$-critical, a contradiction. Since $v(\hat{G}_2) < v(\hat{G})-1$, by the choice of $\hat{G}$, it follows that $\rho(\hat{G}_2) \leq 1$. We then define $\hat{G}_3 = \hat{G}_2 - xy$. Note that $\hat{G}_3 \subsetneq \hat{G}$ and $\rho(\hat{G}_3) = \rho(\hat{G}_2) + 2 \leq 3$. Since $v(\hat{G}_3)\leq v(\hat{G})-2$, it follows that $\hat{G}\neq \hat{G}_3$ and $\hat{G}\neq P_2(\hat{G}_3)$. As $xy \not\in E(\hat{G}_3)$, $\hat{G}_3\not\in \{K_1, K_3\}$. The existence of $\hat{G}_3$ is a contradiction to Lemma~\ref{lem: main lemma}.
\end{itemize}
This completes the proof. 
\end{proof}

\begin{lemma}\label{lem: forbidden triangles with 31}
$\hat{G}$ contains no triangle $T$ of the following type.
\begin{enumerate}[(i)]
    \item\label{case: 31-31-4} two $3_1$-vertices and a $4$-vertex;
    \item\label{case: 31-30-30} a $3_1$-vertex and two $3$-vertices.
\end{enumerate}
\end{lemma}

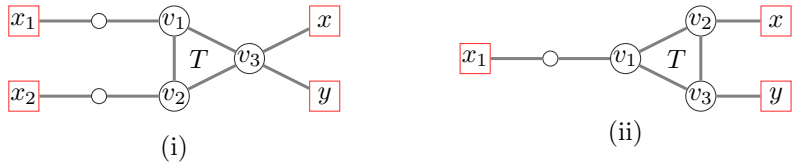
\begin{figure}[ht]
    \centering
	\begin{tikzpicture}[>=latex,
	roundnode/.style={circle, draw=black!80, inner sep=0mm, minimum size=4mm},
	squarednode/.style={rectangle, draw=red!80, inner sep=0mm, minimum size=4mm},
	dotnode/.style={circle, draw=black!80, inner sep=0mm, minimum size=2mm},
	scale=1] 
	
	\begin{scope}[xshift=0cm, yshift = 0cm]
	\node [squarednode] (x) at (4, 1.5){\footnotesize $x$};
	\node [squarednode] (y) at (4, 0.5){\footnotesize $y$};
	\node [roundnode] (v) at (3, 1){\footnotesize $v_3$};
	\node [roundnode] (v1) at (2, 1.5){\footnotesize $v_1$};
	\node [dotnode] (v11) at (1, 1.5){};
	\node [squarednode] (v1') at (0, 1.5){\footnotesize $x_1$};
	\node [roundnode] (v2) at (2, 0.5){\footnotesize $v_2$};
	\node [dotnode] (v22) at (1, 0.5){};
	\node [squarednode] (v2') at (0, 0.5){\footnotesize $x_2$};
	\draw [line width =0.4mm, gray, -] (v2')--(v22)--(v2)--(v)--(x); 
	\draw [line width =0.4mm, gray, -] (v1')--(v11)--(v1)--(v)--(y);
	\draw [line width =0.4mm, gray, -] (v1)--(v2);
	\node at (2,-0.2) {\footnotesize (\ref{case: 31-31-4})};
	\node at (2.33,1) {\footnotesize $T$};
	\end{scope}
	
	\begin{scope}[xshift=6cm, yshift = 0cm]
	\node [squarednode] (x) at (4, 1.5){\footnotesize $x$};
	\node [squarednode] (y) at (4, 0.5){\footnotesize $y$};
	\node [roundnode] (v) at (3, 1.5){\footnotesize $v_2$};
	\node [roundnode] (v1) at (2, 1){\footnotesize $v_1$};
	\node [dotnode] (v11) at (1, 1){};
	\node [squarednode] (v1') at (0, 1){\footnotesize $x_1$};
	\node [roundnode] (v2) at (3, 0.5){\footnotesize $v_3$};
	\draw [line width =0.4mm, gray, -] (v2)--(v)--(x); 
	\draw [line width =0.4mm, gray, -] (v1')--(v11)--(v1)--(v);
	\draw [line width =0.4mm, gray, -] (v1)--(v2)--(y);
	\node at (2,-0) {\footnotesize (\ref{case: 31-30-30})};
	\node at (2.68,1) {\footnotesize $T$};
	\end{scope}
	
	\end{tikzpicture}
	\caption{The two cases in Lemma \ref{lem: forbidden triangles with 31}}
	\label{fig:3_1-3_1-4}
\end{figure}

\begin{proof}
We proceed by cases. In each case, suppose for contradiction that the described triangle $T$ does exist in $\hat{G}$. Let the vertices of $T$ be labeled as is Figure \ref{fig:3_1-3_1-4}. Note that as each $T$ contains a $3_1$-vertex, by Lemma~\ref{lem: 31 in triangle}, $T$ is a positive triangle. Let $\mathcal{P}$ denote the set of paths in $\hat{G}$ drawn in Figure \ref{fig:3_1-3_1-4} that join $v_i$ and one of $x_1,x_2,x,y$ in $\hat{G} - E(T)$. By Lemma~\ref{lem: forbidden theta subgraphs}, there is no edge connecting vertices $x$ and $y$ in $\hat{G}$.

\medskip
\noindent
{\bf (\ref{case: 31-31-4}).}  We add one edge $xy$ and assign it a signature such that $xv_3y$ is a negative triangle. The resulting signed graph is denoted by $\hat{G}'$, i.e., $\hat{G}'=\hat{G}+xy$. Let $\hat{H}$ denote the signed graph obtained from $\hat{G}'$ by deleting $V(T)$, and $v_1$ and $v_2$'s $2$-neighbors.

First, we claim that $\hat{H} \not\rightarrow C_3^*$. Suppose for contradiction there is a homomorphism $\psi: \hat{H} \rightarrow C_3^*$. Let $\sigma$ be a signature of $\hat{G}'$ such that $\psi$ is edge-sign preserving. As in the proof of Lemma~\ref{lem: C_3* crit forbidden triangles with 31}, we may assume the edges of $T$ are positive, and that at most two of the paths in $\mathcal{P}$ are negative under $\sigma$. Furthermore, by possibly switching on the set $V(T)$, we may assume that the two paths of length $2$ in $\mathcal{P}$ are not both negative. We consider two possibilities based on the sign of $xy$:
\begin{itemize}
\item If the edge $xy$ is positive under $\sigma$, then because $v_3yx$ is a negative triangle we may assume without loss of generality that $v_3x$ is negative and $v_3y$ is positive. Moreover, if there is another negative path in $\mathcal{P}$ (except $v_3x$), then by symmetry it is the $v_1x_1$-path. Define $S_1 = V(C_3^*) \setminus \psi(x_1)$, $S_2 = V(C_3^*)$, and $S_3=\{\psi(x)\}$. By Lemma~\ref{lem:triangle}, we can choose $\psi(v_i) \in S_i$, for $i \in \{1,2,3\}$, such that $\psi: V(\hat{H}) \cup V(\hat{T}) \to V(C_3^*)$ is an edge-sign preserving homomorphism and then by Observation~\ref{obs:PathExtension} we may extend $\psi$ to a homomorphism of $\hat{G}$ to $C^*_3$. Note that this is possible because $\psi(x) \neq \psi(y)$ since $\sigma(xy) = +$, and so the ends of the edge $v_3y$ map to different vertices of $C_3^*$. This contradicts that $\hat{G}$ is $C_3^*$-critical. 

\item If $xy$ is negative under $\sigma$, then $\sigma(v_3x) = \sigma(v_3y)$ and $\psi(x) = \psi(y)$. If $\sigma(v_3x) = -$, then $x_1v_1$-path and $x_2v_2$-path are both positive, and in this case we define $S_1 = S_2 = V(C_3^*)$ and $S_3 = \{\psi(x)\}$. If $\sigma(v_3x) = +$, then by our assumptions there is at most one negative path in $\mathcal{P}$, by symmetry say it is $x_1v_1$-path. In this case we define $S_1 = V(C_3^*) \setminus \psi(x_1)$, $S_2 = V(C_3^*)$, and $S_3 = V(C_3^*) \setminus \psi(x)$. In either case, by Lemma~\ref{lem:triangle}, we may choose $\psi(v_i) \in S_i$ such that $\psi: V(\hat{H}) \cup V(\hat{T}) \to V(C_3^*)$ is an edge-sign preserving homomorphism and then, again, by Observation~\ref{obs:PathExtension} we may extend $\psi$ to $\hat{G}$, a contradiction.
\end{itemize}
Therefore, $\hat{H} \not\to C_3^*$.

Since $\hat{H} \not\to C_3^*$, we know that $\hat{H}$ contains a $C_3^*$-critical subgraph $\hat{H}_1$. Moreover, $\hat{H}_1$ must contain the edge $xy$ because otherwise $\hat{H}_1 \subsetneq G$, contradicting that they are both $C_3^*$-critical. Noting that $v(\hat{H}_1)<v(\hat{G})-1$, by the minimality of $\hat{G}$, $\rho(\hat{H}_1)\leq 1$. Let $\hat{H}_2=\hat{H_1}-xy$ and note that $\rho(\hat{H}_2)=\rho(\hat{H}_1)+2\leq 3$. As $\hat{H}_2$ is a subgraph of $\hat{G}$, by Lemma~\ref{lem: main lemma}, either $\hat{H}_2=\hat{G}$, or $\hat{G}=P_2(\hat{H}_2)$, or $\hat{H}_2\in \{K_1, K_3\}$, but all of these are impossible. The first two cannot be because five vertices were deleted in $\hat{G}$ to form $\hat{H}$, and the last because of the deleted edge $xy$.

\medskip
\noindent
{\bf (\ref{case: 31-30-30}).} As in Case (\ref{case: 31-31-4}), add one edge $xy$ but assign it a signature so that $v_2v_3yx$ is a positive 4-cycle. Denote the resulting signed graph by $\hat{G}'$. Let $\hat{H}$ be the signed graph obtained from $\hat{G}'$ by deleting the vertices of $T$ and the 2-vertex adjacent to $v_1$.

First, we claim that $\hat{H} \not\rightarrow C_3^*$. Suppose for contradiction there is a homomorphism $\psi: \hat{H} \rightarrow C_3^*$. Let $\sigma$ be a signature of $\hat{G}'$ which admits $\psi$ as edge-sign preserving. Once again, we may assume that all the edges of $T$ are positive and that at most one of the paths in $\mathcal{P}$ is negative under $\sigma$. We consider two cases based on the sign of $xy$.
\begin{itemize}
\item If $xy$ is negative under $\sigma$, then since $v_2v_3yx$ is a positive 4-cycle we may assume without loss of generality that $v_2x$ is the only negative path in $\mathcal{P}$.
Define $S_1 = V(C_3^*)$, $S_2 = \{\psi(x_2)\}$, and $S_3 = V(C_3^*) \setminus \{\phi(y)\}$. By Lemma~\ref{lem:triangle}, we may choose $\psi(v_i) \in S_i$ such that $\psi: V(\hat{H}) \cup V(\hat{T}) \to V(C_3^*)$ is an edge-sign preserving homomorphism and then, again, by Observation~\ref{obs:PathExtension} we may extend $\psi$ to $\hat{G}$, a contradiction. 

\item If $xy$ is positive under $\sigma$, then $v_2x$ and $v_3y$ are also positive, and $\psi(x) \neq \psi(y)$. Define $S_i = V(C_3^*) \setminus \phi(v_i)$ for $i = 1,2,3$. Then $|S_i| = 2$ for each $i$, but $S_2 \cup S_3 = V(C_3^*)$. Hence by Lemma~\ref{lem:triangle} we can choose $\psi(v_i) \in S_i$ such that $\psi: V(\hat{H}) \cup V(\hat{T}) \to V(C_3^*)$ is an edge-sign preserving homomorphism and then extend $\psi$ to $\hat{G}$ by Observation~\ref{obs:PathExtension}, a contradiction.
\end{itemize}
Therefore, $\hat{H} \not\to C_3^*$.

Since $\hat{H} \not\rightarrow C_3^*$, this means $\hat{H}$ has a $C_3^*$-critical subgraph $\hat{H_1}$, which contains the edge $xy$. This leads to a contradiction in the same manner as Case~(\ref{case: 31-31-4}), and we do not repeat the details.

\medskip
This completes the proof of the lemma.
\end{proof}

\begin{lemma}\label{lem:4_2}
Let $v$ be a $4$-vertex of $\hat{G}$ with two $2$-neighbors $u$ and $w$. Suppose that $x$ and $y$ are the other neighbors of $v$. Then either $xvy$ is a positive triangle, or $xvy$ is a path in a negative $4$-cycle. 
\end{lemma}

\begin{figure}[ht]
    \centering
	\begin{tikzpicture}[>=latex,
	roundnode/.style={circle, draw=black!80, inner sep=0mm, minimum size=4mm},
	squarednode/.style={rectangle, draw=red!80, inner sep=0mm, minimum size=4mm},
	dotnode/.style={circle, draw=black!80, inner sep=0mm, minimum size=2mm},
	scale=1] 
	\node [squarednode] (x) at (4, 1.5){\footnotesize $x$};
	\node [squarednode] (y) at (4, 0.5){\footnotesize $y$};
	\node [roundnode] (v) at (3, 1){\footnotesize $v$};
	\node [roundnode] (v1) at (2, 1.5){\footnotesize $u$};
	\node [squarednode] (v1') at (1, 1.5){\footnotesize $u'$};
	\node [roundnode] (v2) at (2, 0.5){\footnotesize $w$};
	\node [squarednode] (v2') at (1, 0.5){\footnotesize $w'$};
	\draw [line width =0.4mm, gray, -] (v2')--(v2)--(v)--(x); 
	\draw [line width =0.4mm, gray, -] (v1')--(v1)--(v)--(y);
	\end{tikzpicture}
	\caption{A $4$-vertex with two $2$-neighbors.}
	\label{fig:4_2vertex}
\end{figure}
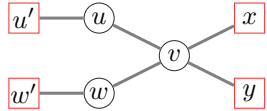

\begin{proof}
Let $u'$ and $w'$ be the neighbors of $u$ and $w$ respectively that are not $v$. See Figure~\ref{fig:4_2vertex}. Let $\sigma$ = $\sigma(\hat{G})$. Suppose for contradiction that $xvy$ is neither in a positive triangle nor a negative $4$-cycle. By possibly switching, we may assume that $\sigma(vx) = \sigma(vy)$ and $\sigma(xy)=-$ if there exists one such edge. If $v' \neq v$ is a common neighbor of $x$ and $y$, then it must be that also $\sigma(v'x) = \sigma(v'y)$ because otherwise, $xvyv'$ is a negative 4-cycle. Let $\hat{G}_1$ be the signed graph obtained from $\hat{G}$ by identifying $x$ and $y$ to a new vertex $z$ and deleting one of the parallel edges connecting $z$ and $v$. 

Since any common neighbor $v'$ of $x$ and $y$ has $\sigma(v'x) = \sigma(v'y)$, the graph $\hat{G}_1$ has no digon. Since $\sigma(xy)=-$, if such an edge exists, then $\hat{G}_1$ may have a negative loop but no positive loop. If there is a homomorphism $\psi: \hat{G}_1 \rightarrow C_3^*$, then $\psi$ can be extended to $\hat{G}$ by setting $\psi(x) = \psi(y) = \psi(z)$. This contradicts that $\hat{G}$ is $C_3^*$-critical. Thus $\hat{G}_1 \not\rightarrow C_3^*$, and hence $\hat{G}_1$ contains a $C_3^*$-critical subgraph $\hat{G}_2$. Note that $z \in V(\hat{G}_2)$ because otherwise $\hat{G}_2 \subsetneq \hat{G}$, but both are $C_3^*$-critical, a contradiction.

Now we claim that the vertex $v$ is not in $\hat{G}_2$. Otherwise, $v$ is a vertex of one of the following types: $1$-vertex, $2_1$-vertex, or $3_2$-vertex, contradicting Lemma~\ref{lem: forbidden configs} or \ref{lem: no 32 vertex}. Note that $u$ and $w$ are also not in $\hat{G}_2$ since $\hat{G}_2$ must be connected and contains no cut-edge by Lemma~\ref{lem: forbidden edge}. Clearly, $v(\hat{G}_2) < v(\hat{G}) - 1$, so $\rho(\hat{G}_2) \leq 1$ by the minimality of $\hat{G}$.

We now construct a signed graph $\hat{G}_3$ from $\hat{G}_2$ as follows: firstly, undo the identification at $z$, without putting back the negative edge $xy$ if such an edge exists; secondly, add the vertex $v$ and edges $vx$ and $vy$. Note that $\hat{G}_3$ is a proper subgraph of $\hat{G}$ and $\rho(\hat{G}_3) = \rho (\hat{G}_2) + 6 -4 \leq 3$. Since $\hat{G}_3 \subsetneq \hat{G}$, by Lemma~\ref{lem: main lemma}, one of the following conditions is satisfied: (1) $\hat{G} = \hat{G}_3$, (2) $\hat{G} = P_2(\hat{G}_3)$, (3) ${G}_3 = K_1$, or (4) ${G}_3 = K_3$. But (1) and (2) are not possible because $u,w \in V(\hat{G}) \setminus V(\hat{G}_3)$, and (3) and (4) are not satisfied because $x,y \in V(\hat{G}_3)$ but are not adjacent. This contradiction completes the proof.
\end{proof}

The next lemma shows us that every $4_3$-vertex of $\hat{G}$ is in exactly three $4$-cycles, all of which share exactly one common edge.

\begin{lemma}\label{lem:43 in 4cycles}
Let $v$ be a $4_3$-vertex of $\hat{G}$, and let $w$ be its neighbor which is not of degree $2$. Then the distance-two neighbors of $v$ are distinct, and $w$ is adjacent to each of them.
\end{lemma}

\begin{figure}[ht]
    \centering
	\begin{tikzpicture}[>=latex,
	roundnode/.style={circle, draw=black!80, inner sep=0mm, minimum size=4mm},
	squarednode/.style={rectangle, draw=red!80, inner sep=0mm, minimum size=4mm},
	dotnode/.style={circle, draw=black!80, inner sep=0mm, minimum size=2mm},
	scale=1] 
	
	\begin{scope}[xshift=0cm, yshift = 0cm]
	\node [squarednode] (x) at (5, 1.6){\footnotesize $x_1$};
	\node [squarednode] (x3) at (5, 1){\footnotesize $x_2$};
	\node [squarednode] (x4) at (5, 0.4){\footnotesize $x_3$};
	\node [roundnode] (v32) at (4, 1.6){\footnotesize $v_1$};
	\node [roundnode] (v33) at (4, 1){\footnotesize $v_2$};
	\node [roundnode] (v34) at (4, 0.4){\footnotesize $v_3$};
	\node [roundnode] (v) at (3, 1){\footnotesize $v$};
	\node [squarednode] (w) at (2, 1){\footnotesize $w$};
	\draw [line width =0.4mm, gray, -] (w)--(v)--(v32)--(x); 
	\draw [line width =0.4mm, gray, -] (x3)--(v33)--(v)--(v34)--(x4);
	\end{scope}
	
	\node (ghost1) at (6,1){$\to$};
	
	\begin{scope}[xshift=5cm, yshift = 0cm]
	\node [squarednode] (x1) at (5, 1.6){\footnotesize $x_1$};
	\node [squarednode] (x2) at (5, 1){\footnotesize $x_2$};
	\node [squarednode] (x3) at (5, 0.4){\footnotesize $x_3$};
	\node [roundnode] (v32) at (4, 1.6){\footnotesize $v_1$};
	\node [roundnode] (v33) at (4, 1){\footnotesize $v_2$};
	\node [roundnode] (v34) at (4, 0.4){\footnotesize $v_3$};
	\node [roundnode] (v) at (3, 1){\footnotesize $v$};
	\node [squarednode] (w) at (2, 1){\footnotesize $w$};
	\draw [line width =0.4mm, gray, -] (w)--(v)--(v32)--(x1); 
	\draw [line width =0.4mm, gray, -] (x2)--(v33)--(v)--(v34)--(x3);
	\draw [line width =0.4mm, gray, -] (x2) to[out=160,in=20] (w) to[out=60,in=160] (x1);
	\draw [line width =0.4mm, gray, -] (w) to[out=-60,in=-160] (x3);
	\end{scope}
	
	\end{tikzpicture}
	\caption{A $4_3$-vertex and its neighborhood.}
	\label{fig:4_3}
\end{figure}
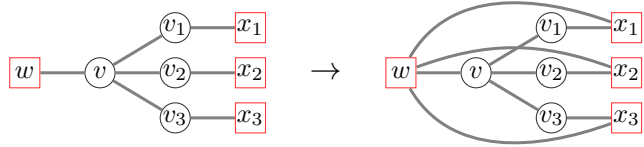

\begin{proof}
Let $v \in V(\hat{G})$ be a $4_3$-vertex whose neighborhood is labeled as in Figure~\ref{fig:4_3}. Since $\hat{G}$ contains no copy of $\Theta_2$, by Lemma~\ref{lem: forbidden theta subgraphs}, $x_1,x_2,$ and $x_3$ are not identified. By Lemma~\ref{lem:4_2}, for each $i$, either $w=x_i$ or $wx_i$ is an edge. First, we show that $w \neq x_i$ for any $i$. If there are distinct $i$ and $j$ such that $w=x_i=x_j$, then $\{w, v, v_i, v_j\}$ induces a copy of $\Theta_1$, contradicting Lemma~\ref{lem: forbidden theta subgraphs}. 
If there are $i$ and $j$ such that $w=x_i$ and $wx_j$ is an edge, then $\{w, v, v_i, v_j, x_j\}$ induces a copy of $\Theta_3$, again contradicting Lemma~\ref{lem: forbidden theta subgraphs}. Therefore, $w \neq x_i$ for any $i$, which means $wx_1, wx_2$ and $wx_3 \in E(\hat{G})$. 
Finally, if there exist distinct $i,j$ such that $x_i = x_j$, then $\{w,v,v_i,v_j,x_i\}$ induces a copy of $\Theta_2$, a contradiction.
\end{proof}

\begin{lemma}\label{lem:$0-43-43$}
\begin{enumerate}[(1)]
    \item \label{43_0} No $5_{\ge2}$-vertex is adjacent to a $4_3$-vertex in $\hat{G}$.
    \item\label{43_1} No $4_0$-vertex is adjacent to two $4_3$-vertices in $\hat{G}$.
    \item\label{43_2} No $5$-vertex is adjacent to four $4_3$-vertices in $\hat{G}$.
\end{enumerate}
\end{lemma}

\begin{figure}[ht]
    \centering
    \begin{minipage}[t]{.46\textwidth}
    \centering
	\begin{tikzpicture}[>=latex,
	roundnode/.style={circle, draw=black!80, inner sep=0mm, minimum size=4mm},
	squarednode/.style={rectangle, draw=red!80, inner sep=0mm, minimum size=4mm},
	dotnode/.style={circle, draw=black!80, inner sep=0mm, minimum size=2mm},
	scale=1] 
	
	\begin{scope}[xshift=0cm, yshift = 0cm]
	\node [roundnode] (w) at (0, 0){\footnotesize $w$};
	\node [roundnode] (v) at (1, 1){\footnotesize $v$};
	\node [dotnode] (1) at (1.5,.5){};
	\node [dotnode] (2) at (2.0,.5){};
	\node [squarednode] (x1) at (2, 0){\footnotesize $x_1$};
	\node [squarednode] (x2) at (3, 0){\footnotesize $x_2$};
	\node [dotnode] (3) at (1.5,-.5){};
	\node [dotnode] (4) at (2.0,-.5){};
	\node [roundnode] (x3) at (1, -1){\footnotesize $x_3$};
    \node [dotnode] (5) at (1,0){};

	\draw [line width =0.4mm, gray, -] (v)--(w)--(x3)--(3)--(x1)--(1)--(v)--(5)--(x3)--(4)--(x2)--(2)--(v); 
	\draw [line width =0.4mm, gray, -] (x1) to[out=160,in=20] (w) to[out=80,in=110, distance=1.6cm] (x2);
	\end{scope}
	
	\end{tikzpicture}
	\caption{The subgraph $H$ which arises when a $4_0$-vertex $w$ is adjacent to two $4_3$-vertices $v$ and $x_3$.}
	\label{fig:4_0}
    \end{minipage}\qquad
    \begin{minipage}[t]{.46\textwidth}
    \centering
	\begin{tikzpicture}[>=latex,
	roundnode/.style={circle, draw=black!80, inner sep=0mm, minimum size=4mm},
	squarednode/.style={rectangle, draw=red!80, inner sep=0mm, minimum size=4mm},
	dotnode/.style={circle, draw=black!80, inner sep=0mm, minimum size=2mm},
	scale=1] 
	
	\begin{scope}[xshift=0cm, yshift = 0cm]
	\node [roundnode] (w) at (0, 0){\footnotesize $w$};
	\node [roundnode] (v) at (2, 1){\footnotesize $x_1$};
	\node [dotnode] (1) at (1.5,.5){};
	\node [dotnode] (2) at (2.5,.5){};
	\node [roundnode] (x1) at (1, 0){\footnotesize $x_2$};
	\node [roundnode] (x2) at (3, 0){\footnotesize $x_3$};
	\node [dotnode] (3) at (1.5,-.5){};
	\node [dotnode] (4) at (2.5,-.5){};
	\node [roundnode] (x3) at (2, -1){\footnotesize $x_4$};
    \node [dotnode] (5) at (3.7,0){};
    \node [dotnode] (6) at (2,0){};
    \node [squarednode] (ghost) at (-1,0){\footnotesize $w'$};

    \draw [line width =0.4mm, gray, -] (ghost) -- (w);
	\draw [line width =0.4mm, gray, -] (v)--(w)--(x3)--(3)--(x1)--(1)--(v)--(5)--(x3)--(4)--(x2)--(2)--(v); 
	\draw [line width =0.4mm, gray, -] (x1) -- (w) to[out=20,in=160, distance=1.0cm] (x2) -- (6) -- (x1);
	\end{scope}
	
	\end{tikzpicture}
	\caption{The subgraph which arises when a $5$-vertex $w$ is adjacent to four $4_3$-vertices $x_1$ to $x_2$.}
	\label{fig:5_3}
	\end{minipage}
\end{figure}

\begin{proof}
(\ref{43_0}) Let $v$ be the $4_3$-vertex and its neighborhood be labeled as in Figure~\ref{fig:4_3}. Suppose for contradiction $w$ is a $5_{\ge 2}$-vertex. Then some $x_i$ must be a $2$-vertex, but this contradicts Lemma~\ref{lem: forbidden configs}, which forbids a $2_1$-vertex in $\hat{G}$.

(\ref{43_1}) Suppose for contradiction such a $4_0$-vertex exists. Let $v$ be one of its $4_3$-neighbors, and let the neighborhood of $v$ be labeled as in Figure~\ref{fig:4_3}. The $4_0$-vertex is clearly $w$, and by Lemma~\ref{lem:43 in 4cycles}, the other $4_3$-vertex must be some $x_i$, say $x_3$. Applying Lemma~\ref{lem:43 in 4cycles} to the $4_3$-vertex $x_3$, it follows that $x_1$ and $x_2$ are at distance $2$ from $x_3$, as in Figure~\ref{fig:4_0}. But then the signed graph $\hat{H}$ induced by $v,w,x_1,x_2,x_3$ and the five $2$-neighbors of $v$ and $x_3$ has $10$ vertices and $14$ edges. This means $\rho(H) =2$. By Lemma~\ref{lem: main lemma}, it must be that $\hat{H} = \hat{G}$. But this is impossible because $\hat{H}$ has a $3_2$-vertex $x_1$, contradicting Lemma~\ref{lem: no 32 vertex}.

(\ref{43_2}) Suppose for contradiction that such a $5$-vertex $w$ exists. Let $x_1,x_2,x_3,$ and $x_4$ be its $4_3$-neighbors, and let $w'$ be the remaining neighbor of $w$. By Lemma~\ref{lem:43 in 4cycles}, for each $i$, the $3$ distance-two neighbors of $x_i$ are among $\{w',x_j : j \neq i\}$. Now, by parity, $w'$ is adjacent to $0,2$, or $4$ of the $x_i$'s. If $w'$ is adjacent to $2$ (or $4$) of them, then the subgraph $\hat{H}$ induced by $\{w,w',x_1,x_2,x_3,x_4\}$ and the $2$-neighbors of each $x_i$ has $\rho(\hat{H}) = 1$ (or $0$, respectively). Both contradict Lemma~\ref{lem: main lemma}. This means each $x_i$ is pairwise connected by a path of length two, where the internal vertex of each path has degree $2$, and $w'$  is adjacent to none of $w_i$'s, as depicted in Figure~\ref{fig:5_3}. But then the edge $w'w$ is a cut-edge in $\hat{G}$, contradicting Lemma~\ref{lem: forbidden edge}.
\end{proof}

\subsection*{Discharging part}
We define a \emph{wealthy vertex} to be a vertex of one of the following types: a $4_0$-vertex, a $5_{\leq 2}$-vertex, or a $6^+$-vertex, and a \emph{rich vertex} to be a vertex of one of the following types: a $4_1$-vertex, a $5_3$-vertex, or a wealthy vertex. 

In the next lemma, we show that every $3_1$-vertex has a ``good'' neighbor.

\begin{lemma}\label{lem:31+rich or wealthy}
Every pair of adjacent $3_1$-vertices has a common wealthy neighbor and every $3_1$-vertex has a rich neighbor.
\end{lemma}

\begin{proof}
Let $v$ be a $3_1$-vertex, and $w,u$ be its neighbors of degree at least $3$. By Lemma~\ref{lem: 31 in triangle}, $w$ and $u$ are adjacent. 

First, suppose that $v$ has a $3_1$-neighbor, say $u$. In this case, we shall show that $w$ is wealthy. By Lemma~\ref{lem: forbidden triangles with 31}, $w$ has degree at least $5$. If $w$ has degree exactly $5$, then $w$ is a $5_{\le 3}$-vertex because of its neighbors $v$ and $u$. But by Lemma~\ref{lem: C_3* crit forbidden triangles with 31}~(\ref{caseC: 31-31-53}), $w$ is not a $5_3$-vertex. This means $w$ is either a $5_{\le 2}$-vertex or a $6^+$-vertex, so $w$ is wealthy by definition.

Now suppose that $v$ does not have a $3_1$-neighbor. In this case, we shall show that one of $w$ or $u$ is rich. If either is of degree at least $5$, then it is rich. So we may assume each of $w$ and $u$ is of degree at most $4$. Suppose that one of them is of degree $3$, say $w$. By Lemma~\ref{lem: forbidden triangles with 31}~(\ref{case: 31-30-30}), $u$ has degree~$4$. By Lemma~\ref{lem: C_3* crit forbidden triangles with 31}~(\ref{caseC: 31-30-42}), $u$ is a $4_{\le 1}$-vertex, so $u$ is rich. So we may assume each of $w$ and $u$ is of degree exactly $4$. Then by Lemma~\ref{lem: C_3* crit forbidden triangles with 31}~(\ref{caseC: 31-42-42}), one of them is a $4_{\le 1}$-vertex, which is rich.
\end{proof}

Then we show that every $4_3$-vertex also has a ``good" neighbor. 

\begin{lemma}\label{lem:43Wealthy}
Every $4_3$-vertex is adjacent to a (unique) wealthy vertex $w$ in $\hat{G}$. Moreover, if $w$ is in a triangle, then it is of degree at least $6$.
\end{lemma}

\begin{proof}
Let $v$ be a $4_3$-vertex, $w$ its non-$2$-neighbor, and whose neighborhood is otherwise labeled as in Figure~\ref{fig:4_3}. Since $2_1$-vertices are forbidden by Lemma~\ref{lem: forbidden configs}, for $i\in \{1,2,3\}$ each $x_i$ has degree at least $3$. By Lemma~\ref{lem:43 in 4cycles}, $w$ is adjacent to each distinct $x_i$. This means $w$ is of degree at least $4$ and moreover, none of its four labeled neighbors is of degree $2$. By the definition, $w$ is a wealthy vertex.

For the moreover part, observe that if $w$ is in a triangle and is of degree at most $5$, then there must be a copy of $\Theta_3$, contradicting Lemma~\ref{lem: forbidden theta subgraphs}.
\end{proof}

Now we are ready to apply the discharging method. We begin with every vertex having a charge equal to its degree. Thus $$\sum\limits_{v\in V(\hat{G})}c(v)=\sum\limits_{v\in V(\hat{G})}d(v)=2e(\hat{G})\leq 3v(\hat{G})-2.$$ We then apply the following three discharging rules to each vertex $v$ in $\hat{G}$.
\begin{enumerate}[(Rule 1)]
    \item\label{R1} {\it Every $3^+$-vertex gives a charge of $\frac{1}{2}$ to each of its $2$-neighbors.}
    \item\label{R2} {\it Every rich vertex gives a charge of $\frac{1}{2}$ to each of its $3_1$-neighbors.}
    \item\label{R3} {\it Every wealthy vertex gives a charge of $\frac{1}{2}$ to each of its $4_3$-neighbors.}
\end{enumerate}
We claim that after discharging, each vertex has a charge of at least $3$, thus $$\sum\limits_{v\in V(\hat{G})}c(v)=\sum\limits_{v\in V(\hat{G})}c'(v)\geq 3v(\hat{G}),$$ a contradiction. 

Let $v \in V(\hat{G})$. We consider three cases based on the type of $v$.

\medskip
\noindent
{\bf Case (1)}. Assume that $v$ is not rich. This means $v$ is one of the following types: a $2$-vertex, a $3$-vertex, a $4_{\ge2}$-vertex, or a $5_{\ge4}$-vertex. The vertex $v$ may give charge only to its $2$-neighbors (if exists) via Rule~\ref{R1}, and may receive charge via Rule~\ref{R2} or \ref{R3}.

If $v$ has degree $2$, then by Lemma~\ref{lem: forbidden configs}, there is no $2_1$-vertex and $v$ has two neighbors of degree at least $3$. By Rule~\ref{R1}, $v$ receives a charge of $\frac{1}{2}$ from each of them. And since $v$ gives no charge to any neighbor, after discharging $v$ has a charge of $3$.

If $v$ is one of $3_0$-vertex, $4_2$-vertex, or $5_4$-vertex, then it has sufficient charge to give to each of its $2$-neighbors (if exists) such that it ends with a charge of $3$. 

By Lemma~\ref{lem: no 32 vertex}, there exists no $3_2$-vertex and by lemma~\ref{lem: forbidden configs}, there exists no $4_4$-vertex and no $5_5$-vertex. Hence, the only remaining cases are when $v$ is a $3_1$-vertex or $4_3$-vertex. By Lemmas~\ref{lem:31+rich or wealthy} and \ref{lem:43Wealthy}, via Rule~\ref{R2} or \ref{R3}, each of these vertices receives a charge of at least $\frac{1}{2}$ from some rich or wealthy neighbor. Therefore, after giving a charge $\frac{1}{2}$ to each of its $2$-neighbors and receiving a charge $\frac{1}{2}$ from its rich or wealthy neighbor, it ends up with a charge of $3$.

\medskip
\noindent
{\bf Case (2)}. Assume that $v$ is rich but not wealthy. This means $v$ is either a $4_1$-vertex or $5_3$-vertex, and thus it may only give charge to its neighbors via Rules~\ref{R1} and \ref{R2}. By Lemma~\ref{lem: 31 in triangle}, if $v$ is adjacent to a $3_1$-vertex $x$, then it is in a triangle with $x$. By Corollary~\ref{cor:unique triangle}, $v$ is in at most one triangle. If $v$ is adjacent to two $3_1$-vertices, say $x,x'$, then $v,x,x'$ are in a triangle. But by the moreover part of Lemma~\ref{lem:31+rich or wealthy}, this means $v$ is wealthy, a contradiction. Hence, $v$ is adjacent to at most one $3_1$-vertex.
This means $v$ gives a charge of at most $\frac{1}{2}$ via Rule~\ref{R2}. Thus after giving a charge $\frac{1}{2}$ to its $2$-neighbors via Rule~\ref{R1}, it still has a charge of at least $3$.

\medskip
\noindent
{\bf Case (3)}. Assume that $v$ is wealthy. This means $v$ is one of the following types: a $6^+$-vertex, a $4_0$-vertex, or a $5_{\leq 2}$-vertex. If $v$ is a $6^+$-vertex, then by Rules~\ref{R1}, \ref{R2}, and \ref{R3}, after discharging it always has a charge of $d(v)-\frac{1}{2}d(v)\geq 3$.

Hence, we may assume $v$ is either a $4_0$-vertex, or a $5_{\leq 2}$-vertex. Note that by Lemma~\ref{lem:43Wealthy}, Rules~\ref{R2} and \ref{R3} cannot apply at the same time to $v$. By Corollary~\ref{cor:unique triangle}, $v$ is in at most one triangle and can therefore give a charge of at most $1$ via Rule~\ref{R2}. Thus after applying Rules~\ref{R1} and \ref{R2}, $v$ is left with a charge of at least $3$. It remains only to consider when Rule~\ref{R3} is applied. By Lemma~\ref{lem:$0-43-43$} (\ref{43_0}), if $v$ is a $5_2$-vertex, it has no $4_3$-neighbor, and so gives no charge under Rule \ref{R3}. By Lemma~\ref{lem:$0-43-43$}~(\ref{43_1}), $v$ is adjacent to at most one $4_3$-vertex if it is a $4_0$-vertex, and by Lemma~\ref{lem:$0-43-43$}~(\ref{43_2}), $v$ is adjacent to at most three $4_3$-vertices if it is a $5_{\leq 1}$-vertex. This means after applying Rules~\ref{R1} and \ref{R3}, $v$ still has a charge of at least $3$.

We are done.

\section{Tightness and discussion}\label{sec:Conclusion}

We have seen in Lemma~\ref{lem:hatW} that $\hat{W}$ is a $C^*_3$-critical signed graph satisfying that $e(\hat{W})= \frac{3v(\hat{W})-1}{2}$. Now we provide a sequence of $C^*_3$-critical signed graphs with edge density $\frac{3}{2}$, showing that the bound in Theorem~\ref{thm:main result} is asymptotically tight.

Let $\hat{S}_{k}$ be the connected signed graph obtained from a negative cycle of length $k$ by adding a set $S$ of $k$ vertices and $2k$ edges so that each edge of the negative cycle is contained in a positive triangle with a distinct vertex in $S$, as in Figure \ref{fig:couterK8}. 
It is easy to observe that $e(\hat{S}_k)= \frac{3v(\hat{S}_k)}{2}.$ Note that in any edge-sign preserving homomorphism of a switching-equivalent $\hat{S}_k$ to $C_3^*$, every edge of each triangle needs to be positive. Hence, such a signature does not exist because $\hat{S}_k$ has a negative $k$-cycle which requires at least one negative edge appearing in one triangle. Therefore $\hat{S}_k \not\to C_3^*$. Furthermore, there is a signature of $\hat{S}_k$ with exactly two negative edges (both in the same triangle). If any edge $e$ is deleted, then by symmetry we may assume it is one of those two negative edges, and with this signature $\hat{S}_k- e \to C_3^*$. So $\hat{S}_k$ is $C_3^*$-critical.

\begin{figure}[h]
	\centering
	\begin{minipage}[t]{.34\textwidth}
	\centering
	\begin{tikzpicture}[>=latex,
	roundnode/.style={circle, draw=black!80, inner sep=0mm, minimum size=3.5mm},
        roundnode2/.style={circle, draw=black!80, inner sep=0mm, minimum size=3mm},
	scale=0.76] 
	\node [roundnode] (A) at (0:2cm){};
		\node[roundnode2] (AB) at (18:2.67cm){};
	\node [roundnode] (B) at (36:2cm){};
		\node[roundnode2] (BC) at (54:2.67cm){};
	\node [roundnode] (C) at (72:2cm){};
		\node[roundnode2] (CD) at (90:2.67cm){};
	\node [roundnode] (D) at (108:2cm){};
		\node[roundnode2] (DE) at (126:2.67cm){};
	\node [roundnode] (E) at (144:2cm){};
		\node[roundnode2] (EF) at (162:2.67cm){};
	\node [roundnode] (F) at (180:2cm){};
		\node[roundnode2] (FG) at (198:2.67cm){};
	\node[roundnode] (G) at (216:2cm){};
		\node[roundnode2] (GH) at (236:2.67cm){};
	\node[roundnode] (H) at (252:2cm){};
	\node [roundnode] (I) at (288:2cm){};
		\node[roundnode2] (IJ) at (306:2.67cm){};
	\node [roundnode] (J) at (324:2cm){};
		\node[roundnode2] (JA) at (342:2.67cm){};
	
	\foreach \t in {18,54,90,126,162,198,234,306,342}{	
		\node at (\t:2.2cm){$+$};}
	\node at (0:0cm){$-$};
			
	\draw[line width = 0.4mm, gray] (A)--(B)--(C)--(D)--(E)--(F)--(G)--(H);
	
	\draw[line width = 0.4mm, gray] (A)--(AB)--(B)--(BC)--(C)--(CD)--(D)--(DE)--(E)--(EF)--(F)--(FG)--(G)--(GH)--(H);
	
	\draw[line width = 0.4mm, gray] (I)--(IJ)--(J)--(JA)--(A);

	\draw[line width = 0.4mm, dotted, gray] (H)--(I);
	\draw[line width = 0.4mm, gray]	(I)--(J)--(A);
	\end{tikzpicture}
	\caption{Signed graph $\hat{S}_{k}$}
	\label{fig:couterK8}
	\end{minipage}
\begin{minipage}[t]{.65\textwidth}
\centering
   \begin{subfigure}[t]{.45\textwidth}
   \centering
		\begin{tikzpicture}
		[>=latex,
	roundnode/.style={circle, draw=black!80, inner sep=0mm, minimum size=4mm},scale=0.34] 
		\foreach \i in {1,2,3,4,5}
		{
		\draw[rotate=72*(\i-1)] (0, 5) node[roundnode] (x_\i){\scriptsize ${v_{\i}}$};
		}
		
		\foreach \i in {1,2,3,4,5}
		{
		\draw[rotate=72*(\i-1)] (0, 3) node[roundnode] (y_\i){\scriptsize ${u_{\i}}$};
		}
		
		\foreach \i/\j in {1/2, 2/3, 3/4, 4/5, 5/1}
		{
		\draw [line width =0.4mm, blue, -] (x_\i) -- (x_\j);
		}
		
		\foreach \i in {1,2,3,4,5}
		{
		\draw [line width =0.4mm, blue, -] (x_\i) -- (y_\i);
		}
		
		\foreach \i/\j in {1/3, 3/5, 5/2, 2/4, 4/1}
		{
		\draw  [densely dotted, line width =0.4mm, red, -] (y_\i) -- (y_\j);
		}
        \end{tikzpicture}
		\end{subfigure}
       \begin{subfigure}[t]{.45\textwidth}
       \centering
		\begin{tikzpicture}
		[>=latex,
	roundnode/.style={circle, draw=black!80, inner sep=0mm, minimum size=4mm},scale=0.34] 
		\foreach \i in {1,2,3,4,5}
		{
		\draw[rotate=72*(\i-1)] (0, 5) node[roundnode] (x_\i){\scriptsize ${v_{\i}}$};
		}
		
		\foreach \i in {1,2,3,4,5}
		{
		\draw[rotate=72*(\i-1)] (0, 3) node[roundnode] (y_\i){\scriptsize ${u_{\i}}$};
		}
		
		\foreach \i/\j in {1/2, 3/4, 5/1}
		{
		\draw [line width =0.4mm, blue, -] (x_\i) -- (x_\j);
		}

                \foreach \i/\j in {2/3, 4/5}
		{
		\draw [densely dotted, line width =0.4mm, red, -] (x_\i) -- (x_\j);
		}
		
		\foreach \i in {1,2,3,4,5}
		{
		\draw [line width =0.4mm, blue, -] (x_\i) -- (y_\i);
		}
		
		\foreach \i/\j in {2/5}
		{
			\draw  [densely dotted, line width =0.4mm, red, -] (y_\i) -- (y_\j);
		}
                \foreach \i/\j in {1/3, 3/5, 2/4, 4/1}
		{
			\draw  [line width =0.4mm, blue, -] (y_\i) -- (y_\j);
		}
       \end{tikzpicture}
		\end{subfigure}
		\caption{Signed  Petersen graph $\hat{P}$}
		\label{fig:Petersen}
	\end{minipage}
\end{figure} 

Next, we shall prove that the girth bound in Corollary~\ref{cor:Planar+ProjectivePlanar} for the class of signed projective-planar graphs is tight. We show that there is a signed projective-planar graph of girth $5$ which does not admit a homomorphism to $C^*_3$ in the next lemma. 

Let $\hat{P}$ be a signed Petersen graph with a signature such that the edges of one $5$-cycle are negative and all the other edges are positive, depicted in Figure~\ref{fig:Petersen}. We note that $\hat{P}$ is a signed projective-planar graph of girth $5$.

\begin{lemma}\label{lem:Petersen}
The signed graph $\hat{P}$ is $C^*_3$-critical.
\end{lemma}

\begin{proof}
It is easy to see that for any edge $e$, $\hat{P}-e$ satisfies that
$\rho(\hat{P}-e)=2$. Hence, by Theorem~\ref{thm:main result alt} and because of the degrees of the vertices of $\hat{P}-e$, it
does not contain any $C^*_3$-critical signed graph as a subgraph and thus it admits a homomorphism to $C^*_3$.

Proving that $\hat{P}$ does not admit a homomorphism to $C^*_3$ is a bit more technical.  
For a contradiction, assume that there exists such a mapping. First, observe that any mapping of a positive $5$-cycle to $C^*_3$ covers all positive edges of the target while it is the opposite for negative $5$-cycles (at least one positive edge of $C^*_3$ is not used as the homomorphic image of an edge of a negative $5$-cycle). By the
high symmetry of the signed Petersen graph, observe also that any negative $5$-cycle can be isomorphically seen as the central cycle, denoted by $S=u_1u_3u_5u_2u_4$, depicted in
Figure~\ref{fig:Petersen}. Let us study how a negative $5$-cycle can be mapped to $C^*_3$ and let $a,b$ and $c$ denote the vertices of $C^*_3$.

\begin{enumerate}[(1)]
\item All vertices $u_1$ up to $u_5$ of the negative cycle $S$ are mapped to one vertex of $C^*_3$, say $a$. Since $v_1u_1u_4u_2v_2$ is a positive $5$-cycle, it should cover all the positive edges of the target $C^*_3$ so $v_1$ and $v_2$ should be mapped to $b$ and $c$ respectively. The same argument stands for the pairs $v_2v_3$, $v_3v_4$, $v_4v_5$, and $v_5v_1$. This is impossible since the $5$-cycle is not bipartite. So no negative $5$-cycle has all its vertices mapped to a single vertex of $C^*_3$.

\item Four vertices of $S$, say $u_1$ up to $u_4$, are mapped to one vertex of $C^*_3$, say $a$ and the vertex $u_5$ is mapped to $b$. Then, considering the positive cycles $u_1u_3v_3v_4u_4$ and $u_1v_1v_2u_2u_4$ whose homomorphic images should cover all the positive edges of the target signed graph $C^*_3$, we may conclude that $\{v_1,v_2\}$ is mapped to $\{b,c\}$ and so does $\{v_3,v_4\}$. Now if $v_1$ is mapped to $b$ and $v_2$ to $c$, then the homomorphism image of the negative $5$-cycle $v_1v_2u_2u_5v_5$ already covers all positive edges of $C^*_3$ which is a contradiction with our first observation. The same result holds for the pair $\{v_3,v_4\}$. Thus it implies that both $v_1$ and $v_4$ are mapped to $c$ while both $v_2$ and $v_3$ are mapped to $b$. This prevents the positive (outer) $5$-cycle $v_1v_2v_3v_4v_5$ to be mapped to $C^*_3$, as wherever $v_5$ is mapped it will not cover the edge $ab$. So no negative $5$-cycle has four of its vertices mapped to a single vertex of $C^*_3$.
 
 \item Three consecutive of $S$, say $u_3,u_1$ and $u_4$, are mapped to $a$. In this case, both other vertices $u_2$ and $u_5$ must be mapped to the same vertex, say $b$, to preserve the sign of the cycle. Since $v_3u_3u_1u_4v_4$ is a positive $5$-cycle, we may assume without loss of generality that $v_3$ is mapped to $b$ while $v_4$ is mapped to $c$. But then the edges of the negative $5$-cycle $u_2v_2v_3v_4u_4$ already cover all positive edges of the target $C^*_3$ which is a contradiction. So no three consecutive vertices of a negative $5$-cycle can be mapped to a single vertex of $C^*_3$.  
\end{enumerate}

Most of the technical part is done. Now observe that any three
consecutive vertices (i.e., vertices that induce a path) of $\hat{P}$ are involved in a negative $5$-cycle. Together with our previous case study, this implies that no three vertices of a $3$-path can be mapped to a single vertex of
$C^*_3$. In other words, any switching-equivalent signature $\pi$ of $\hat{P}$ that allows an
edge-sign preserving homomorphism of $(P, \pi)$ to $C^*_3$ does not have two incident negative edges. Up to isomorphism, only one switching-equivalent signature $\pi$ of $\hat{P}$ achieves this and it is on the right of Figure~\ref{fig:Petersen}. If such an edge-sign preserving homomorphism was possible, then $v_2$ and $v_3$ are mapped to $a$, $v_4$ and $v_5$ are mapped to $b$, and $u_2$ and $u_5$ are mapped to $c$. In order to
prevent the vertices of any $3$-path being mapped to a single vertex, it enforces $u_3$ to be mapped to $b$, $u_4$ to $a$, and $v_1$ to $c$. But then $u_1$ cannot be mapped anywhere while preserving the sign of its incident edges. This concludes the proof that $\hat{P}$ does not admit a homomorphism to $C^*_3$.
\end{proof}

Actually, in a personal communication of the fifth author with R. Naserasr and S. Mishra, it has been shown that $\chi_c(\hat{P})=\frac{10}{3}$ which also implies that $\hat{P}\not\to C^*_3$. Here we give the shortest proof that we could achieve for the sake of completeness.

However, we don't know if the girth bound of $6$ in Corollary~\ref{cor:Planar+ProjectivePlanar} is tight for the class of signed planar graphs. 

\medskip
We have seen in Theorem~\ref{thm:C3critical} that if a graph $G$ satisfies that $e(G)<\frac{5v(G)-2}{3}$, then $G\to C_3$. Hence, any graph $G$ with average degree less than $\frac{10}{3}-\frac{4}{3v(G)}$ is $3$-colorable. It implies Brook's theorem when $\Delta=3$, i.e., except $K_4$ (where $v(K_4)=4$), any graph on more than $4$ vertices with maximum degree $3$ is $3$-colorable. The previous argument means that the $C_3$-critical graph with $\Delta=3$ is unique and it is $K_4$. We may consider the analogous problem for $C^*_3$-critical signed graphs. It has been proved in Theorem~\ref{thm:main result} that if a signed graph $\hat{G}$ satisfies that $e(\hat{G})<\frac{3v(\hat{G})-1}{2}$, then $\hat{G}\to C^*_3$. Thus any signed graph with its average degree less than $3-\frac{1}{v(G)}$ is circular $3$-colorable. We pose the following question:

\begin{problem}
Are there finitely many $C^*_3$-critical signed graphs with $\Delta=3$?
\end{problem}

\paragraph{Acknowledgement.} This work was initiated as a group project at the ANR-HOSIGRA Workshop, held at CAES du CNRS La Villa Clythia in Fr\'ejus, France in May 2022. The authors wish to thank ANR project HOSIGRA (ANR-17-CE40-0022) for providing this support.

The second and third authors were partially supported by the Natural Sciences and Engineering Research Council of Canada (NSERC). The fourth author was partially funded by IFCAM project ``Applications of graph homomorphisms'' (MA/IFCAM/18/39) for this work. The fifth author was partially supported by European Union's Horizon 2020 research and innovation program under the Marie Sklodowska-Curie grant agreement No 754362.

\end{document}